%% file: G2subgroups.tex
\documentclass[12pt]{amsart}
\usepackage{amssymb,latexsym,amsmath,amsthm,amscd}
\usepackage{setspace}
\usepackage[active]{srcltx}
\usepackage[all]{xy} \xyoption{arc}
\usepackage[left=3cm,top=2cm,right=3cm,bottom = 2cm]{geometry}
\usepackage{graphicx}
\usepackage[usenames,dvipsnames]{color}
\usepackage{tikz}
\usepackage{array}
\usepackage{enumitem}
\usepackage{mathtools}
\usetikzlibrary{snakes}
\usepackage{listings}

\usetikzlibrary{arrows}

\usepackage{todonotes}
\usepackage{charter}
\usepackage{hyperref}
\hypersetup{
  colorlinks  = true, 
  urlcolor    = blue, 
  linkcolor   = blue, 
  citecolor   = blue 
}

\theoremstyle{plain}
\newtheorem{thm}{Theorem}[section]
\newtheorem{lem}[thm]{Lemma}
\newtheorem{prop}[thm]{Proposition}

\theoremstyle{definition}

\theoremstyle{remark}
\newtheorem{rmk}[thm]{Remark}

\input{preamble}

\newcommand{\OO}{\mathbf{O}}
\newcommand{\HH}{\mathbf{H}}
\DeclareMathOperator{\short}{short}
\DeclareMathOperator{\llong}{long}

\DeclareMathOperator{\SU}{SU}
\DeclareMathOperator{\Spin}{Spin}

\DeclareMathOperator{\U}{U}

\DeclareMathOperator{\surj}{surj}
\DeclareMathOperator{\Id}{Id}
\DeclareMathOperator{\Gr}{Gr}

\begin{document}
\title{The moduli space of representations of the modular group into $G_2$}

\author{Angelica Babei}
\address{Howard University}
\email{angelica.babei@howard.edu}
\author{Andrew Fiori}
\address{University of Lethbridge}
\email{andrew.fiori@uleth.ca}
\author{Cameron Franc}
\address{McMaster University}
\email{franc@math.mcmaster.ca}

\thanks{The authors gratefully acknowledge financial support received from NSERC through their respective Discovery Grants, the financial support of the University of Lethbridge, and the use of computational resources made available through WestGrid and Compute Canada.}
\date{}

\begin{abstract}
  In this paper we construct a large four-dimensional family of representations of the modular group into $G_2$. Precisely, this family is an etale cover of degree $96$ of an open subset of the moduli space of such representations. This moduli space has two main components, of dimensions one and four. The one-dimensional component consists of well-studied rigid representations, in the sense of Katz. We focus on the four-dimensional component which consists of representations that are not rigid. We also provide algebraic conditions to ensure that the specializations surject onto $G_2(\mathbf{F}_p)$ for primes $p\geq 5$. These representations give new examples of $\phi$-congruence subgroups of the modular group as introduced in \cite{bff}.
\end{abstract}
\maketitle

\setcounter{tocdepth}{1}
\tableofcontents

\section{Introduction}
In this paper we construct a large family of representations, $\phi$ , of the modular group $\Gamma = \PSL_2(\ZZ)$ into $G_2$. Such a representation is determined by choosing elements of $G_2$ of order two and three, and mapping respective free-generators
\begin{align*}
S &= \twomat {0}{-1}{1}{0}, &R &= \twomat{0}{-1}{1}{1},
\end{align*}
for $\Gamma$ to these elements. We are particularly interested in representations that surject onto $G_2(\FF_p)$ for primes $p$. This is motivated by our earlier paper \cite{bff} where we introduced $\phi$-families of subgroups of $\Gamma$ by reducing linear representations $\phi$ of $\Gamma$ modulo primes. In \cite{bff} we considered two main examples: a family arising from an indecomposable but not irreducible representation of rank two, as well as families arising from maps into the symplectic group of genus two. The symplectic case is interesting because by results of \cite{LS}, most choices of order two and three elements in classical groups lead to surjective representations of $\Gamma$, whereas for the $g=2$ symplectic group, if $p$ is large enough then only approximately half the choices yield surjective representations. Thus, the symplectic group is somewhat exceptional and it is harder to construct representations onto it.

In \cite{bff} we were able to write down a large supply of such symplectic representations by using the description of the character variety of rank four irreducible representations of $\Gamma$ described by Tuba-Wenzl \cite{TW}. By \cite{TW}, character varieties of $\Gamma$ are well-understood up to rank five, and Le Bruyn has written a number of interesting papers on higher rank cases, see for example \cite{lebruyn}. Thus, the case of representations of $G_2$ of rank seven is an interesting exceptional case of examples that lies just outside of the known literature of representations of the modular group, but which intersects the active area of study of $G_2$-local systems.

As for the growing literature on $G_2$-local systems, much of it has focused on \emph{rigid} representations of the full free group $F_2$ on two generators, a subject introduced by Katz \cite{katz}. See for example the papers \cite{CDRS,DR1,DR2} for some work on this subject. Boalch and Paluba \cite{Boalch1} have studied $G_2$-character varieties obtained from the fundamental group of a four-punctured sphere, where three of the conjugacy classes of local monodromies are fixed to consist of elements of order $3$ in $G_2$, whereas the fourth local monodromy can be a more general semisimple conjugacy class. As this fourth conjugacy class is allowed to vary, Boalch-Paluba show that the resulting two-parameter family of complex surfaces is related to a subfamily of the Fricke family of cubic surfaces
\[
xyz + x^2+y^2+z^2+b_1x+b_2y+b_3z=0.
\]
These examples of $G_2$-character varieties parameterize local systems that are minimally nonrigid in the sense of Katz, and so provide a suitably rich class of examples for studying local systems beyond the rigid case. In a similar vein, our examples below correspond to character varieties of a three-punctured sphere where two local monodromies are fixed to lie in conjugacy classes of order $2$ and $3$, and the third class is allowed to vary.

More precisely, as we will see below, there are two disconnected components of irreducible representations $\Gamma \to G_2$, one of which is one-dimensional while the other is four-dimensional. It transpires that the one-dimensional component consists of rigid representations, and none of these lead to surjections onto $G_2(\FF_p)$. On othe other hand, the larger component $X$ consists of non-rigid representations and to the best of our knowledge, it has not received as much study as the rigid component. Therefore, in our analysis below, we will ignore the more well-studied rigid component and we focus our attention on the non-rigid component $X$.\footnote{Technically, our variety $X$ of representations will contain many representations that are not irreducible, though the irreducibles are dense in $X$. We use the adjective ``irreducible'' in this introduction mainly to rule out trivial components such as the trivial representation, et cetera.}

Our approach to studying $X$ is to realize it concretely as a double-quotient of $G_2$ of the following form:
\[
X = \GL_2^{\short}\backslash G_2/\SO_{\HH}.
\]
In this optic, we are considering simultaneous conjugacy classes of the generating elements of $\Gamma$ of orders $2$ and $3$. Ultimately we shall prove the following:
\begin{thm}
  \label{t:main1}
  There exists a nonempty open subset $Z \subseteq X$ in the large component and a cover $\widetilde X_0 \to Z$ described concretely by the family of representations:
\begin{align*}
\phi(S) &= \left(\begin{smallmatrix}
a_{1}^{2} - a_{2}^{2} & 0 & 0 & 0 & 2 a_{1} a_{2} & 0 & 0 \\
0 & b_{1}^{2} - b_{2}^{2} & 0 & 0 & 0 & 2 b_{1} b_{2} & 0 \\
0 & 0 & c_{1}^{2} - c_{2}^{2} & 0 & 0 & 0 & 2 c_{1} c_{2} \\
0 & 0 & 0 & -1 & 0 & 0 & 0 \\
2 a_{1} a_{2} & 0 & 0 & 0 & -a_{1}^{2} + a_{2}^{2} & 0 & 0 \\
0 & 2 b_{1} b_{2} & 0 & 0 & 0 & -b_{1}^{2} + b_{2}^{2} & 0 \\
0 & 0 & 2 c_{1} c_{2} & 0 & 0 & 0 & -c_{1}^{2} + c_{2}^{2}
\end{smallmatrix}\right),\\
  \phi(R)  &= \left(\begin{smallmatrix}
1 & 0 & 0 & 0 & 0 & 0 & 0 \\
0 & 1 & 0 & 0 & 0 & 0 & 0 \\
0 & 0 & 1 & 0 & 0 & 0 & 0 \\
0 & 0 & 0 & -\frac{1}{2}& -\frac{3}{2}( 1 -2y_1) & -3 x_1  & 3 w_1 \\
0 & 0 & 0 & \frac{1}{2}(1 - 2y_1) & -\frac{1}{2} & w_1 & 3x_1\\
0 & 0 & 0 & 3x_1 & -3w_1 & -\frac{1}{2} & -\frac{3}{2}( 1 -2y_1) \\
0 & 0 & 0 & -w_1 & -3 x_1 & \frac{1}{2}( 1 - 2y_1) & -\frac{1}{2}
\end{smallmatrix}\right).
\end{align*}
where
\begin{enumerate}
\item $a_1^2+a_2^2=b_1^2+b_2^2=c_1^2+c_2^2 = 1$,
\item $w_1^2+3x_1^2+y_1^2=y_1$, and
 \item $\stwomat {a_1}{-a_2}{a_2}{a_1}\stwomat {b_1}{-b_2}{b_2}{b_1}\stwomat {c_1}{-c_2}{c_2}{c_1} = I$.
 \end{enumerate}
 This map $\widetilde X_0 \to Z$ is generically finite of degree $48$. The Galois group of this cover is generated by the permutations of the pairs $(a_1,a_2)$, $(b_1,b_2)$, $(c_1,c_2)$ in addition to sign changes of the pairs that respect equation (3).

 If we restrict to specializations where $a_1b_1c_1 \neq 0$ and set
   \begin{align*}
  t(\phi) &=-3\frac{(a_2b_1c_1)^2+(a_1b_2c_1)^2+(a_1b_1c_2)^2}{(a_1b_1c_1)^2},\\
  u(\phi) &= -3\frac{a_2b_2c_1+a_2b_1c_2+a_1b_2c_2}{a_1b_1c_1},
   \end{align*}
   and let $\phi_1$ and $\phi_2$ be two specializations of this family, then the corresponding representations are equivalent if and only if we have $t(\phi_1) = t(\phi_2)$, $u(\phi_1)=  u(\phi_2)$, and further the $y_1$ coordinates of the specializations are equal, whereas the $x_1$ and $w_1$ coordinates of the two specializations are equal up to sign.
\end{thm}
Theorem \ref{t:main1} is a consequence of Theorem \ref{thm:genericfiniteall} and a slight reparameterization discussed in Section \ref{s:representation}. See the proof of Theorem \ref{t:main1} in Section \ref{proofmain1}. Note that the map $\widetilde{X}_0 \to Z$ is in particular finite of degree $48$ over the locus of points discussed in Theorem \ref{t:main2} below. The invariants $t(\phi)$ and $u(\phi)$ appearing in Theorem \ref{t:main1} above are defined in Section \ref{s:computational} below, and we have retained the factors of $-3$ to be consistent with that later material.

Before outlining the layout of the paper, we introduce some notation that will both allow us to connect this work to other results in the literature, and which will allow us to list algebraic conditions that will ensure our representation is surjective. If $\phi$ is a representation as in Theorem \ref{t:main1}, let the characteristic polynomial of $\phi(T)$, where $T = \stwomat 1101$ as usual, be
    \[
   X^7+g_1X^6+g_2X^5+g_3X^4-g_3X^3-g_2X^2-g_1X-1.
 \]
 Then $g_3=g_1+g_2-g_1^2$ since $\phi(T)$ is an element of $G_2$. Furthermore, $g_1$ has a simple expression $g_1=3(a_2^2+b_2^2+c_2^2)-5$, while $g_2$ is more complicated to write down. Therefore we will not do so, though a reader can easily compute it using a computer algebra system. Note that this notation for $g_1$ and $g_2$ will only appear in parts (5) and (6) of Theorem \ref{t:main2}.

 Theorem \ref{t:main1} relates to the work of Boalch-Paluba \cite{Boalch1} as follows: regarding $\phi(T)$ as the monodromy at infinity of the local system corresponding to $\phi$, then fixing the conjugacy class of $\phi(T)$ corresponds to fixing the monodromy at infinity, similarly to what is done in \cite{Boalch1}. In particular, since the characteristic polynomial of $\phi(T)$ is invariant under conjugation, this corresponds to fixing $(g_1,g_2)$. In this way, our family $\widetilde{X}_0$ of Theorem \ref{t:main1} is fibered over the pairs $(g_1,g_2)$ above, and this realizes $\widetilde{X}_0$ as a two-parameter family of complex surfaces, similar to the situation in \cite{Boalch1}. In a personal communication, Philip Boalch has suggested to us that this might be a family of H3 surfaces as discussed in his classification of minimally nonrigid character varieties of complex dimension two \cite{Boalch2}. More precisely, he has suggested that the fibers of $\widetilde{X}_0$ over a pair $(g_1,g_2)$ might be affine-$E_6$ examples discussed in Section 4.1 of \cite{Boalch2}, but we have not verified whether this interesting suggestion holds.

Turning from the rich global geometry of the family $\widetilde{X}_0$ introduced in Theorem \ref{t:main1}, we are also interested in the more arithmetic or algebraic conditions on when the representations in our family surject onto $G_2(\FF_p)$. Our next main theorem gives algebraic conditions that classify precisely this: 
\begin{thm}
  \label{t:main2}
  Let $\widetilde X_0^{\surj}$ denote the subset of $\widetilde X_0$ which is the open complement to the following list of conditions:
  \begin{enumerate}
  \item $a_1a_2b_1b_2c_1c_2 =0$;
  \item $(a_1^2-a_2^2)(b_1^2-b_2^2)(c_1^2-c_2^2)=0$;
  \item $wxyz = 0$;
  \item $(x^2-z^2)(w^2-y^2)(3x^2-y^2)(3z^2-w^2)=0$;
  \item $g_1^5-2g_1^3g_2-g_1^3-g_1^2g_2+2g_1g_2^2+g_2^3=0$;
  \item $g_1$ equals one of the following possibilities: if $\alpha$ denotes a root of $\alpha^3-3\alpha+1 = 0$,
    \[
   0,~\pm 1,\pm 2,~-3,~-7,~ \alpha,~ (1/2)(\pm\sqrt{13}-1).
 \]
\item in characteristic $11$, we must also avoid the following values of $(g_1,g_2)$:
  \begin{align*}
    &(1,8),~ (10,9),~ (8,1),~ (4,1),~ (1,10),~ (0,0),~ (6,9),\\
    & (2,6),~ (4,10),~ (6,3),~ (3,2),~(7,10),~ (9,5),~ (4,7).
  \end{align*}
\end{enumerate}
Then all representations $\phi \in \widetilde X_0^{\surj}$ with image in $G_2(\FF_p)$ surject onto $G_2(\FF_p)$, where $p\geq 5$.
\end{thm}

In the preceding Theorem we must avoid $p=2$ and $3$ for a variety of reasons. The conditions of Theorem \ref{t:main2} can be significantly refined by substituting the more precise conditions that can be found throughout Section \ref{s:maximal}. For simplicity we have given a simpler but less comprehensive statement in the Theorem \ref{t:main2}. Presumably this surjectivity could be lifted to a density statement for represenations landing in $G_2(\ZZ_p)$ where $\ZZ_p$ denotes the $p$-adic integers using a Hensel lifting argument as in Section 2.2 of \cite{bff}, but we do not pursue this here.

The paper is organized as follows. In Section \ref{s:background} we introduce background on the octonions and $G_2$. In Section \ref{s:moduli} we begin the study of the geometry of $X$ by first analyzing the intermediate quotient $G_2/\SO_{\HH}$. In Section \ref{p: LargePatch} we begin by analyzing the invariant theory for the $\GL_2^{\short}$-action on an open subset of $G_2/\SO_{\HH}$. Ultimately in Section \ref{p: LargePatch} we construct and establish the stated properties of the varieties underlying Theorem \ref{t:main1}. Section \ref{s:representation} introduces our family of representations, which is essentially a versal family living over the moduli space. Finally, in Section \ref{s:maximal} we establish the surjectivity claims of Theorem \ref{t:main2}.

If $\phi \colon \Gamma \to G_2(\QQ)$ is a representation as above whose reduction $\phi_p$ mod $p$ surjects onto $G_2(\FF_p)$, then the kernel $\ker \phi_p$ is noncongruence. In this paper we say nothing about the corresponding noncongruence modular forms. Nor do we say anything about the local systems on $\PP^1\setminus\{0,1,\infty\}$ corresponding to $\phi$ beyond the details of the local monodromy. Other open problems related to this work include: extending our family to other patches in the moduli space as outlined in Section \ref{ss:othercells}; identifying the fibers of $\widetilde{X}_0$ over a fixed conjugacy class of $\phi(T)$ within the classification outlined by Boalch \cite{Boalch2}; determining whether the global space $\widetilde{X}_0$ can be related to some well-known family of varieties such as the Fricke family arising in \cite{Boalch1}. We hope to return to some of these questions in future work.

We thank Philip Boalch for comments on an earlier draft of this manuscript.
\subsection{Notation}
\begin{itemize}
\item[---] $p$ is a prime with $p\geq 5$.
\item[---] $K$ denotes a field.
\item[---] $\Gamma = \PSL_2(\ZZ)$.
\item[---] $\HH$ is a rational quaternion algebra such that $i^2=-3$, $j^2=-1$ and $\HH_1$ is the subgroup of norm $1$ elements.
\item[---] $\OO$ is an octonion algebra constructed from $\HH$ via the Cayley-Dickson construction. The trace zero subspace of $\OO$ is denoted $\OO_0$.
\item[---] $G_2$ is the group of automorphisms of $\OO$ acting on $\OO_0$.
\item[---] $\SO_\HH\subseteq G_2$ is a maximal subgroup isomorphic with $\SO_4$ that is described in Section \ref{sss:d2}.
\item[---] $\GL_2^{\llong}$ and $\GL_2^{\short}$ are $\GL_2$-subgroups of $G_2$ that are described in Section \ref{ssec:gl2}.
\item[---] $\SU_M\subseteq G_2$ is an $A_2$-subgroup described in Section \ref{ssec:a2}.
\item[---] $D_6$ is the Weyl group of $G_2$, described in Section \ref{ss:weyl}.
\item[---] $\alpha_3,\alpha_3' \in G_2$ are representatives for the conjugacy classes of order three elements in $G_2$, which are described in Section \ref{ss:conjclasses}.
\item[---] $X$ is the large  connected component of the moduli space of homomorphisms $\Gamma \to G_2(K)$, described in Section \ref{ss:connected}.
\end{itemize}

\section{Background on $G_2$}
\label{s:background}
The model we take for the group $G_2$ is as the automorphisms of an octonion algebra. There is only one form of the octonions, or the group $G_2$, over any algebraically closed field, finite field, or local field other than $\RR$. Consequently, for most considerations the form of the octonions is not so important. Since we will ultimately be interested in elements of $G_2$ of order three, below we use a nonstandard presentation of the octonions so that these elements have simpler integral presentations. For information on octonions and $G_2$ one can consult Chapter VIII, Section \textbf{33.C}, of \cite{BOI}, where octonion algebras are referred to as Cayley algebras.

\subsection{Models for the octonions and $G_2$} The following model for $\OO$ over $\QQ$ reduces mod $p$ to an octonion algebra for primes $p\neq 2,3$. Define $\QQ$-linear generators $i$, $j$ and $k$ such that a $\QQ$-basis for $\OO$ is
\[
  \cB =(1,i,j,ij,k,ik,jk,(ij)k).
\]
Set $i^2=-3$, $j^2=k^2=-1$. The remaining relations are the standard octonion relations. More precisely, the first set of relations state that any two distinct nonidentity elements $x,y \in \cB$ anticommute: $xy=-yx$. The last set of relations are that any three nonidentity basis vectors $x,y,z\in \cB$ that do not lie inside a quaternion algebra satisfy $(xy)z= -x(yz)$.

Given
\[
  x = a_1+a_2i+a_3j+a_4ij+a_5k+a_6ik+a_7jk+a_8(ij)k
\]
with $a_j\in\QQ$ for all $j$, we define the \emph{conjugate}
\[
\bar x = a_1-a_2i-a_3j-a_4ij-a_5k-a_6ik-a_7jk-a_8(ij)k.
\]
This satisfies $\overline{xy} = \bar{y}\bar{x}$. The \emph{trace} is $\Tr(x) = x+\bar x = 2a_1$, and the norm is
\[
  N(x) = x\bar x = (a_1^2+a_3^2+a_5^2+a_7^2)+3(a_2^2+a_4^2+a_6^2+a_8^2).
\]
This is a nondegenerate positive definite quadratic form on $\OO/\QQ$. The trace-zero subspace $\OO_0$ of $\OO$ is also sometimes called the (purely) imaginary octonions. The group $G_2$ is realized below with respect to its action on the trace-zero octonions $\OO_0$.

There is an alternative description of the octonions that we will make use of below called the \emph{Cayley-Dickson construction}. To describe this, let $\HH$ denote a rational quaternion algebra in the usual basis $1,i,j,ij$ such that $i^2=-3$ and $j^2=-1$ as above. The Cayley-Dickson construction arises by writing
\[
  \OO = \HH \oplus \HH
\]
where the multiplication and conjugation are described as:
\begin{align*}
  (u_1,v_1)(u_2,v_2) &= (u_1u_2- \overline{v_2}v_1,  v_2u_1 + v_1\overline{u_2} ),\\
  \overline{(u_1,v_1)} &= (\overline{u_1},-v_1).
\end{align*}
Identifying $\HH \subseteq \OO$ with the first coordinate, and using the standard basis $1,i,j,ij$ of $\HH$, we set $k=(0,1)$ so that $ik=(0,i)$, $jk=(0,j)$ and $(ij)k = -i(jk) = (0,ij)$.

Since $\OO_0$ is a $7$-dimensional rational vector space, our explicit presentation of $G_2$ will be as $7\times 7$-matrices arising via the ordered basis obtained from $\cB$.

\begin{rmk}\label{rem:altsub}
  The algebra $\OO$ has two other useful decompositions of the same form $\HH \oplus \HH$ where multiplication is defined by the Cayley-Dickson construction.
\begin{align*} \OO &= {\rm Span}(1,i,k,ik) \oplus  {\rm Span}(j,ij,kj,(ik)j), \\ \OO &= {\rm Span}(1,i,jk,i(jk)) \oplus  {\rm Span}(k,ik,(jk)k,(i(jk))k). \end{align*}
\end{rmk}

\subsection{Some relevant subgroups of $G_2$}\label{ss:relsub}
To describe the moduli space of representations from the modular group into $G_2$ we will make use of various subgroups of $G_2$. For ease of reference we collect the facts we need in this subsection.

\subsubsection{A $D_2 = A_1\times A_1$-subgroup}
\label{sss:d2}
There is an action of $\SO_4$ on $\OO$ via automorphisms, realizing $\SO_4$ as a subgroup of $G_2$. To describe this group concretely, let $\HH_1$ denote the subgroup of quaternions of norm one. Then there is an isomorphism
\[
  \SO_4 \cong \Spin_3\times \Spin_3/\{\pm 1\}\cong \HH_1\times \HH_1/\{\pm 1\}. 
\]
This leads to a left action of $\SO_4$ on $\OO$ via the Cayley-Dickson construction. More precisely, given $(h_1,h_2) \in \HH_1^2$ and $(u,v) \in \OO = \HH\oplus \HH$, we have
\[
  (h_1,h_2) \cdot (u,v) \df (h_1u\overline{h_1},h_2v\overline{h_1}).
\]
Note that the factor of $\overline{h_1}$ in the second coordinate above is \emph{not} a typo. This action of $\SO_4$ on the trace-zero octonions decomposes into a sum of two irreducible representations, whereby $i$, $j$ and $ij$ span one of the spin representations, and $k$, $ik$, $jk$, $(ij)k$ span the standard representation of $\SO_4$. We shall denote this specific subgroup of $G_2$ by $\SO_\HH$. It is a maximal subgroup of $G_2$, and all $D_2$ subgroups of $G_2$ are conjugate, over an algebraically closed field, to this standard one. This $D_2$ subgroup has an order $2$ centre given by $(-1,1)$ and it is the centralizer in $G_2$ of this central element. Indeed, every element of order $2$ in $G_2$ has centralizer of type $D_2$.

The action of $\HH_1$ on $\HH$ given by $h_1\circ v= v\overline{h_1}$ is described, in the basis $k$, $ik$, $jk$, $(ij)k$, by
\[  \{  a+bi+cj+dij \;|\; a^2+3b^2+c^2+3d^2=1 \} \rightarrow 
\left(\begin{smallmatrix} 
a &3b &c &3d \\
-b & a &-d& c\\
-c& 3d &a &-3b\\
-d& -c&  b&a 
\end{smallmatrix}\right). \]
The action of $\HH_1$ on $\HH$ given by $h_2\circ v= h_2v$ is described, in the basis $k$, $ik$, $jk$, $(ij)k$, by
\[  \{  e+fi+gj+hij \;|\; e^2+3f^2+g^2+3h^2=1 \} \rightarrow 
\left(\begin{smallmatrix}
e &-3f & -g&-3h \\
f &e & -h&  g\\
g &3h &e & -3f\\
h & -g &f &e 
\end{smallmatrix}\right). \]
The action of $\HH_1$ on $\HH_0$ given by $h_1\circ u = h_1u\overline{h_1}$ is described, in the basis $i$, $j$, $ij$, by 
\[  \{  a+bi+cj+dij \;|\; a^2+3b^2+c^2+3d^2=1 \} \rightarrow  
\left(\begin{smallmatrix} 
a^2+3b^2-c^2-3d^2 &-2ad+2bc& 2ac+6bd \\
6ad+6bc &a^2-3b^2+c^2-3d^2& -6ab+6cd \\
-2ac+6bd & 2ab+2cd &a^2-3b^2-c^2+3d^2 
\end{smallmatrix}\right) .\]

\begin{rmk}\label{rmk:altD2}
With the maximal torus as defined  \ref{ssec:mt} below, the subgroup $\HH_1\times \{1\}$ gives two complementary short roots of $G_2$, whereas the subgroup $\{1\} \times \HH_1$ gives two complementary long roots. These roots do not generate the root system of $G_2$ because they are perpendicular.

We can similarly construct conjugates of $\HH_1\times\HH_1$ using the two decompositions from Remark \ref{rem:altsub}. Each gives two complementary short roots and two complementary long roots. Taken together, these roots along with the maximal torus span the Lie algebra of $G_2$.
\end{rmk}

\subsubsection{Maximal torus}
\label{ssec:mt} We can use the action discussed in Section \ref{sss:d2} to describe a maximal torus $\U_1\times \U_1/ \{\pm 1\}$ in $G_2$ by letting the norm one elements
\[\U_1 \df \{a+bi \in \HH \mid a^2+3b^2=1\}\]
act on the imaginary octonions $\OO_0$. Concretely, tracing through our various choices, one sees that an element $(a+bi,c+di) \in \U_1\times \U_1$ acts as
\[
  \left(\begin{smallmatrix}
      1&&&&&&\\
      &a^2-3b^2&-6ab&&&&\\
      &2ab&a^2-3b^2&&&&\\
      &&&ac+3bd&3bc-3ad&&\\
      &&&ad-bc&ac+3bd&&\\
      &&&&&ac-3bd&-3ad-3bc\\
      &&&&&ad+bc&ac-3bd
\end{smallmatrix}\right).
\]
This gives a maximal torus in $G_2$.

\subsubsection{$\GL_2$-subgroups}
\label{ssec:gl2}

We can use the action discussed in Section \ref{sss:d2} to identify two $\GL_2$-subgroups of $G_2$.
Both of these are Levi subgroups of maximal parabolic subgroups of $G_2$ containing the maximal torus of Section \ref{ssec:mt}.

\begin{rmk}
These subgroups are isomorphic to $\GL_2$ over an algebraically closed field.
\end{rmk}

\paragraph{Short root $\GL_2$-subgroup}\label{p:shortroot}

The subgroup $\HH_1\times \U_1/ \{\pm 1\}$  gives what we will call the short root $\GL_2$ subgroup.

The centre of this $\GL_2$ subgroup is the image of $1 \times \U_1$, and this subgroup is the centralizer in $G_2$ of its centre. Indeed, this subgroup is the centralizer of any element from the image of $1\times \U_1$ of order at least $3$ (the centralizer of an order $2$ element will be of type $D_2$). We shall denote this specific subgroup of $G_2$ by $\GL_2^{\short}$.

 
\paragraph{Long Root $\GL_2$-subgroup}\label{p:longroot}

The subgroup $\U_1\times \HH_1/ \{\pm 1\}$ gives what we will call the long root $\GL_2$ subgroup.

The centre of this $\GL_2$ subgroup is the image of $\U_1 \times 1$, and this subgroup is the centralizer in $G_2$ of its centre. Indeed, this subgroup is the centralizer of any element from the image of $\U_1\times 1$ of order at least $4$ (the centralizer of an order $3$ element from the centre will be of type $A_3$, see below, the centralizer of an order $2$ element will be of type $D_2$). We shall denote this specific subgroup of $G_2$ by $\GL_2^{\llong}$.


\subsubsection{$A_2$-subgroups}
\label{ssec:a2}
There are several ways to identify an $A_2$ subgroup of $G_2$. The first is the subgroup generated by the long roots of $G_2$, that is, the subgroup generated by the long root parts of the three subgroups from Remark \ref{rmk:altD2}.

Alternatively, following  \cite[Section 2]{fiori1}, we see that with $L = {\rm Span}(1,i)$ and $M=L^\perp={\rm Span}(j,ij,k,ik,jk,i(jk))$ we can write
\[ \OO = L \oplus M \]
and recognize $L$ as a quadratic algebra and $M$ as a $3$-dimensional left $L$-module with $L$-basis $j$, $k$ and $jk$. Through polarization, the norm form on $\OO$ induces a Hermitian inner product on $M$ and we can re-interpret the product structure on $\OO$ using the decomposition $\OO = L\oplus M$ as
\[ (\ell_1,m_1)(\ell_2,m_2) = (\ell_1\ell_2 + m_1\cdot m_2, \ell_1\cdot m_2 + \overline{\ell_2}\cdot m_1 + m_1\times m_2), \]
where $m_1\cdot m_2$ is the Hermitian pairing and $m_1\times m_2$ is the associated Hermitian cross product. Indeed, given any Hermitian pairing on a three dimensional module $M$ over a quadratic module $L$ the above defines an octonion algebra. Through functoriality, automorphisms of the Hermitian space $M$ preserving the cross product, that is, elements of $\SU_M$, determine automorphisms of $\OO$, hence they yield elements of $G_2$. We shall denote this precise $A_2$ subgroup of $G_2$ by $\SU_M$.

These automorphisms all take $i$ to $i$, and indeed, this $A_2$ subgroup consists of exactly the automorphisms of $\OO$ taking $i$ to $i$. To see this, notice that the decomposition $\OO=L\oplus M$ is determined by the choice of an element of $\OO_0$ with non-zero norm  and the Hermitian structure is determined by the multiplication on $\OO$. Over an algebraically closed field every $A_2$ subgroup of $G_2$ is conjugate and are determined by a choice of non-isotropic line in $\OO_0$.

This $A_2$-subgroup is not maximal, though its normalizer is. The $A_2$ subgroup has index $2$ in its normalizer which contains the order $2$ element taking $i$ to $-i$ and fixing both $j$ and $k$.
This order $2$ element gives the outer automorphism of the $A_2$ subgroup $\SU_M$.
The centre of this $A_2$ subgroup is $\mu_3$, a group of order $3$, and $\SU_M$ is the centralizer in $G_2$ of its centre. The outer automorphism of $\SU_M$ acts non-trivially on the centre of $\SU_M$, and so it interchanges the two order $3$ elements.

\subsection{The Weyl group}
\label{ss:weyl}

The Weyl group for $G_2$ is the (semi-)direct product of the Weyl group of $\SU_M$, which is isomorphic to $S_3$, with the outer-automorphism of the $\SU_M$-subgroup, which is realized in $G_2$ as the map 
\[ \begin{array}{cccc} i \mapsto -i &j \mapsto j & k \mapsto k & jk \mapsto jk \end{array}\]
This acts on the root system as a rotation by $180^\circ$.
This group has order $12$ and is isomorphic to $D_6$. 
The Weyl group of the $\SU_M$ subgroup comes from permutations of the lines spanned by $j$, $k$, and $jk$, that is, representatives for non-trivial elements are given by maps which act on the generators $i$, $j$, and $k$ as follows:
\[\begin{array}{cccc}
i\mapsto i& j\mapsto j & k\mapsto jk  & jk \mapsto -k\\
i\mapsto i& j\mapsto jk & k\mapsto k & jk \mapsto -j \\
i\mapsto i& j\mapsto k & k\mapsto j & jk \mapsto -jk\\
i\mapsto i& j\mapsto k & k\mapsto jk & jk \mapsto j\\
i\mapsto i& j\mapsto jk & k\mapsto j & jk \mapsto k.
\end{array}\]
Note that these generate a group of order $24$, as it contains the subgroup of the maximal torus of Section \ref{ssec:mt} consisting of its order $2$ elements. 
Note that the outer automorphism of $\SU_M$-given above does indeed commute with these as a $180^\circ$ degree rotation commutes with all other rotations and reflections.

As the Weyl group is generated by reflections, one could alternatively describe elements by finding all of the reflections. 
Because each of the three $\SO_4$ subgroups from Remark \ref{rmk:altD2} identifies a pair of perpendicular elements in the root system, we see that the Weyl group of $G_2$ will be generated by the Weyl groups of these three $\SO_4$ subgroups. The first three of these listed above give the reflections for the long root $\SL_2$ factor of these $\SO_4$ subgroups. The order $3$ elements listed give rotations by $120^{\circ}$.

\subsection{Conjugacy classes of some finite order elements}
\label{ss:conjclasses}
Every finite order element, of order not divisible by the characteristic of the field, is conjugate, over an algebraically closed field, to an element of a fixed maximal torus.
It follows that conjugacy classes of finite order elements can be determined by looking at finite order elements of the maximal torus modulo the action of the Weyl group.

In $G_2$ the maximal torus factors through a simply connected $A_2$ subgroup. As there is a unique conjugacy class of order $2$ element in $\SL_2$, the same is true of $G_2$. With the maximal torus as presented in Section \ref{ssec:mt}, we fix as a representative $\alpha_2$ the image of $(-1,1)=(1,-1)\in \U_1\times\U_1/(-1,-1)$. 
Concretely, this gives
\[ \alpha_2 =   \left(\begin{smallmatrix} 1 \\ &1\\ &&1\\ &&&-1\\ &&&&-1\\ &&&&&-1 \\&&&&&&-1 \end{smallmatrix}\right). \]

In $\SL_3$ there are three conjugacy classes of order $3$ elements, two of which are central. The outer automorphism of $\SL_3$, which is an element of $G_2$ and part of the Weyl group of the maximal torus, interchanges these two central elements of $\SL_3$.
It follows that there are $2$ conjugacy classes of order $3$ elements in $G_2$. One which is central in $\SL_3$ and one which is not.
As discussed in Section \ref{p:shortroot} the image, $\alpha_3$, of $(1,\zeta_3)\in \U_1\times\U_1/(-1,-1)$ does not centralize an $\SL_3$ subgroup whereas the image, $\alpha_3'$, of $ (\zeta_3,1)\in \U_1\times\U_1/(-1,-1)$ does.
Noting that $\zeta_3\in \U_1$ is $\tfrac{1}{2}(-1+i)$, we find that
\begin{align*}
 \alpha_3 &= \left(\begin{smallmatrix} 1 \\ &1\\ &&1\\ &&&-\tfrac{1}{2} & -\tfrac{3}{2}\\ &&&\tfrac{1}{2}& -\tfrac{1}{2} \\ &&&&&-\tfrac{1}{2} & -\tfrac{3}{2} \\&&&&&\tfrac{1}{2}& -\tfrac{1}{2} \end{smallmatrix}\right), &
 \alpha_3' &=  \left(\begin{smallmatrix} 1 \\ &-\tfrac{1}{2} &\tfrac{3}{2}\\ &-\tfrac{1}{2}&-\tfrac{1}{2} \\ &&&-\tfrac{1}{2} & \tfrac{3}{2}\\ &&&-\tfrac{1}{2}& -\tfrac{1}{2} \\ &&&&&-\tfrac{1}{2} & -\tfrac{3}{2} \\&&&&&\tfrac{1}{2}& -\tfrac{1}{2}  \end{smallmatrix}\right). 
 \end{align*}

\section{Geometry of the moduli space}
\label{s:moduli}
Let $G$ be a group. As a representation of $\Gamma$ is determined by the choice of an order $2$ and order $3$ element of $G$ up to simultaneous conjugacy by $G$, the moduli space of representations $\Gamma \to G$ is a disjoint union indexed by conjugacy classes of such elements. Fixing representatives $\alpha_2$ and $\alpha_3$ for these conjugacy classes in $G$, the corresponding conjugacy classes are:
\begin{align*}
 C_{\alpha_2} &= \{ g\alpha_2g^{-1} \;|\; g\in G  \}  \cong     G / Z_{G}(\alpha_2),\\
 C_{\alpha_3} &= \{ g\alpha_3g^{-1} \;|\; g\in G  \}  \cong     G / Z_{G}(\alpha_3).
\end{align*}
With this notation, each component $X$ of the moduli space of representations of $\Gamma$ into $G$ is clearly covered
\[
(G/ Z_G(\alpha_2)) \times (G/ Z_G(\alpha_3)) \to X,
\]
where to a pair $(g_1Z_G(\alpha_2),g_2Z_G(\alpha_3))$ we associate the representation $\rho$ determined by
\begin{align*}
  \rho(S) &= g_1\alpha_2g_1^{-1},\\
  \rho(R) &= g_2\alpha_3g_2^{-1}.
\end{align*}

The map above is equivariant for the natural left action of $G$ on  $(G/ Z_G(\alpha_2)) \times (G/ Z_G(\alpha_3))$ and the conjugation action on representations. Using this, it follows that each element of $X$ has representatives of the form
\[ (  Z_{G}(\alpha_2),  g_2Z_{G}(\alpha_3)) \qquad\text{and}\qquad ( g_1 Z_{G}(\alpha_2),  Z_{G}(\alpha_3)). \]
These representatives are not unique, indeed we have
\begin{align*}
 (  Z_{G}(\alpha_2),  g_2Z_{G}(\alpha_3)) &=  (  Z_{G}(\alpha_2), hg_2Z_{G}(\alpha_3))\qquad \text{ for } h\in Z_{G}(\alpha_2), \\
  ( g_1 Z_{G}(\alpha_2),  Z_{G}(\alpha_3)) &=  (h g_1 Z_{G}(\alpha_2),  Z_{G}(\alpha_3))\qquad \text{ for } h\in Z_{G}(\alpha_3). 
  \end{align*}
This yields two descriptions of $X$:
\[   Z_{G}(\alpha_2)  \backslash  G / Z_G(\alpha_3) \cong X \cong  Z_{G}(\alpha_3)  \backslash  G / Z_G(\alpha_2). \] 
Using the first identification above, given $Z_G(\alpha_2)gZ_G(\alpha_3)$, points of $X$ correspond to representations $\rho$ satisfying
\begin{align*}
  \rho(S) &= \alpha_2,\\
  \rho(R) &= g\alpha_3g^{-1}.
\end{align*}
The second description associates to $Z_G(\alpha_3)gZ_G(\alpha_2)$ the representation $\rho$ satisfying
\begin{align*}
  \rho(S) &= g\alpha_2g^{-1},\\
  \rho(R) &= \alpha_3.
\end{align*}
Below in our analysis of the large component $X$ we shall use this latter identification of the component. What we ultimately construct is a finite etale cover of a Zariski dense open subset of $X$.

\subsection{Connected components}
\label{ss:connected}
Because there are two conjugacy class options for $\alpha_3$ and only one option for $\alpha_2$, cf. Section \ref{ss:conjclasses}, we identify two components of the moduli space of representations $\Gamma \to G$. For reasons which shall become apparent, we shall refer to the case  $\alpha_3 =  (1,\zeta_3)$ as the large component, and $\alpha_3' = (\zeta_3,1)$ as the small component. See Section \ref{ss:conjclasses} for a description of this notation.

By applying the discussion from Section \ref{s:background} and the start of Section \ref{s:moduli}, we can realize these components as
\[  X = \GL_2^{\short} \backslash G_2 / \SO_{\HH} \qquad \text{and}\qquad \SU_M \backslash G_2 / \SO_{\HH}. \]
We will see that $X$ is four-dimensional. That $\SU_M \backslash G_2 / \SO_{\HH}$ is one-dimensional is well-known  \cite{CDRS, DR1, DR2}.

To study the larger component $X$, it is convenient to first study the intermediate quotient $G_2/\SO_{\HH}$, and then look at the action of $\GL_2^{\short}$ on this quotient.

The smaller component $\SU_M \backslash G_2 / \SO_{\HH}$ was studied in \cite{CDRS, DR1, DR2} using Katz's theory of rigid local systems \cite{katz}. By contrast, if one were to study $\SU_M \backslash G_2 / \SO_{\HH}$ using a  group theoretic approach, it would be convenient to write
\[\SU_M \backslash G_2 / \SO_{\HH} \cong  \SO_{\HH} \backslash G_2 / \SU_M\]
and study first the intermediate quotient $G_2 / \SU_M$. We do not pursue this in this paper and instead focus on $X$.

\subsection{The quotient $G_2 / \SO_{\HH}$}
\label{ss:G2modSOH}
Since $\SO_\HH$ is the stabilizer of the quaternion subalgebra spanned by $1$, $i$, $j$, $ij$, and $G_2$ acts transitively on these subalgebras (as follows from the Cayley-Dickson construction and the fact that there is a unique form of quaternion algebra over an algebraically closed field), the quotient $G_2/\SO_\HH$ is identified with the set of quaternion subalgebras of $\OO$. These subalgebras are in bijection with three-dimensional subspaces $V\subseteq \OO_0$, where the multiplication map has image in the span of $1$ and $V$, and where the norm form restricts to a non-degenerate pairing on this subspace. The multiplication condition is a Zariski-closed condition, whereas the non-degeneracy of the quadratic form is a Zariski-open condition. This shows that the quotient $G_2/\SO_\HH$ is a quasi-projective variety that can be realized as a subvariety of the $(3,7)$-grassmanian.

We will work in the affine chart for the $(3,7)$-grassmanian defined by the span of the columns of matrices of the form
\[ \left(\begin{matrix} 
1 & 0 & 0 \\
0 & 1 & 0 \\
0 & 0 & 1 \\
a_1 & b_1 & c_1 \\
a_2 & b_2 & c_2 \\
a_3 & b_3 & c_3 \\
a_4 & b_4 & c_4
\end{matrix}\right).\]
That means we will be missing some of the moduli space in our description, but this part of the moduli space will prove to be complicated enough to analyze. In any case, this chart covers a dense portion of $X$.

To define the closed subvariety consisting of subspaces generating $4$-dimensional algebras, using the coordinates introduced above and describing $\OO$ via the Cayley-Dickson construction, all three of the following products 
\begin{align*}
  (i,a_1+a_2i + a_3j+a_4ij)(j,b_1+b_2i + b_3j+b_4ij),\\
  (i,a_1+a_2i + a_3j+a_4ij)(ij,c_1+c_2i + c_3j+c_4ij),\\
  (j,b_1+b_2i + b_3j+b_4ij)(ij,c_1+c_2i + c_3j+c_4ij),
\end{align*}
must be contained in
\[  {\rm Span}((1,0), (i,a_1+a_2i + a_3j+a_4ij),(j,b_1+b_2i + b_3j+b_4ij),(ij,c_1+c_2i + c_3j+c_4ij)). \]
Explicit algebraic equations for this condition can be obtained by augmenting the standard chart above with additional columns corresponding to these octonion products:
\begin{equation}\label{augmat1}
\left(\begin{smallmatrix}
1 & 0 & 0 & -a_{2} b_{1} + a_{1} b_{2} + a_{4} b_{3} - a_{3} b_{4} & -a_{2} c_{1} + a_{1} c_{2} + a_{4} c_{3} - a_{3} c_{4} & -b_{2} c_{1} + b_{1} c_{2} + b_{4} c_{3} - b_{3} c_{4} + 1 \\
0 & 1 & 0 & -a_{3} b_{1} - 3 a_{4} b_{2} + a_{1} b_{3} + 3 a_{2} b_{4} & -a_{3} c_{1} - 3 a_{4} c_{2} + a_{1} c_{3} + 3 a_{2} c_{4} - 3 & -b_{3} c_{1} - 3 b_{4} c_{2} + b_{1} c_{3} + 3 b_{2} c_{4} \\
0 & 0 & 1 & -a_{4} b_{1} + a_{3} b_{2} - a_{2} b_{3} + a_{1} b_{4} + 1 & -a_{4} c_{1} + a_{3} c_{2} - a_{2} c_{3} + a_{1} c_{4} & -b_{4} c_{1} + b_{3} c_{2} - b_{2} c_{3} + b_{1} c_{4} \\
a_{1} & b_{1} & c_{1} & a_{3} - 3 b_{2} & 3 a_{4} - 3 c_{2} & 3 b_{4} - c_{3} \\
a_{2} & b_{2} & c_{2} & a_{4} + b_{1} & -a_{3} + c_{1} & -b_{3} - c_{4} \\
a_{3} & b_{3} & c_{3} & -a_{1} + 3 b_{4} & 3 a_{2} + 3 c_{4} & 3 b_{2} + c_{1} \\
a_{4} & b_{4} & c_{4} & -a_{2} - b_{3} & -a_{1} - c_{3} & -b_{1} + c_{2}
\end{smallmatrix}\right).
\end{equation}
The rank of this matrix is at least $3$, and our closed condition amounts to the rank being exactly $3$. Thus, the closed condition is the vanishing of all $4\times 4$-minors of this matrix. A groebner basis for this set of conditions can be easily computed. Using these groebner basis elements it is not hard to show that our closed conditions imply the following identities:
\begin{align*}
  &\det\left(\begin{smallmatrix}
      c_1 & b_1 & a_1\\
      c_2 & b_2 & a_2\\
      c_3 & b_3 & a_3 
  \end{smallmatrix}\right)= a_{3} - 3 b_{2} - c_{1},\\
  &\det\left(\begin{smallmatrix}
      c_1 & b_1 & a_1\\
      c_2 & b_2 & a_2\\
      c_4 & b_4 & a_4 
  \end{smallmatrix}\right)= a_{4} + b_{1} - c_{2},\\
&\det\left(\begin{smallmatrix}
      c_1 & b_1 & a_1\\
      c_3 & b_3 & a_3\\
      c_4 & b_4 & a_4 
  \end{smallmatrix}\right) =- a_{1} + 3 b_{4} - c_{3},\\
&\det\left(\begin{smallmatrix}
      c_2 & b_2 & a_2\\
      c_3 & b_3 & a_3\\
      c_4 & b_4 & a_4 
  \end{smallmatrix}\right)=- a_{2} - b_{3} - c_{4},\\
&\det\stwomat {b_1}{a_1}{b_2}{a_2}-\det\stwomat{b_3}{a_3}{b_4}{a_4}=\det\stwomat{c_1}{b_1}{c_4}{b_4}-\det\stwomat{c_2}{b_2}{c_3}{b_3},\\
&\det\stwomat{b_1}{a_1}{b_3}{a_3} + 3\det\stwomat{b_2}{a_2}{b_4}{a_4}= \det\stwomat{c_2}{a_2}{c_3}{a_3}-\det\stwomat{c_1}{a_1}{c_4}{a_4} ,\\
&\det\stwomat{c_1}{b_1}{c_3}{b_3}+ 3\det\stwomat{c_2}{b_2}{c_4}{b_4}=\det\stwomat{c_3}{a_3}{c_4}{a_4}-\det\stwomat{c_1}{a_1}{c_2}{a_2}.
\end{align*}
In particular, the closed conditions are not homogeneous.

One can check that the ideal of $4\times 4$ minors of \eqref{augmat1} is generated by the minors using the first three columns. Note that by performing column operations on the matrix in \eqref{augmat1} using the first three columns, which won't change the ideal generated by the $4\times 4$-minors using the first three columns, we can express our closed conditions equivalently using the $4\times 4$-minors of the following matrix with more symmetric-looking entries:
\begin{equation}
  \label{eq:cM}
\cM \df \left(\begin{smallmatrix}
1 & 0 & 0 & \det \stwomat{a_1}{b_1}{a_2}{b_2}-\det\stwomat{a_3}{b_3}{a_4}{b_4}  & \det\stwomat{a_1}{c_1}{a_2}{c_2}-\det\stwomat{a_3}{c_3}{a_4}{c_4} & \det\stwomat{b_1}{c_1}{b_2}{c_2}-\det\stwomat{b_3}{c_3}{b_4}{c_4}\\
0 & 1 & 0 & \det\stwomat{a_1}{b_1}{a_3}{b_3}+3\det\stwomat{a_2}{b_2}{a_4}{b_4} & \det\stwomat{a_1}{c_1}{a_3}{c_3} +3\det\stwomat{a_2}{c_2}{a_4}{c_4} &\det\stwomat{b_1}{c_1}{b_3}{c_3} +3\det\stwomat{b_2}{c_2}{b_4}{c_4}\\
0 & 0 & 1 &\det\stwomat{a_1}{b_1}{a_4}{b_4} -\det\stwomat{a_2}{b_2}{a_3}{b_3} & \det\stwomat{a_1}{c_1}{a_4}{c_4}-\det\stwomat{a_2}{c_2}{a_3}{c_3} & \det\stwomat{b_1}{c_1}{b_4}{c_4}-\det\stwomat{b_2}{c_2}{b_3}{c_3}\\
a_{1} & b_{1} & c_{1} & a_{3} - 3 b_{2} - c_{1} & 3 a_{4} + 3 b_{1} - 3 c_{2} & -a_{1} + 3 b_{4} - c_{3} \\
a_{2} & b_{2} & c_{2} & a_{4} + b_{1} - c_{2} & -a_{3} + 3 b_{2} + c_{1} & -a_{2} - b_{3} - c_{4} \\
a_{3} & b_{3} & c_{3} & -a_{1} + 3 b_{4} - c_{3} & 3 a_{2} + 3 b_{3} + 3 c_{4} & -a_{3} + 3 b_{2} + c_{1} \\
a_{4} & b_{4} & c_{4} & -a_{2} - b_{3} - c_{4} & -a_{1} + 3 b_{4} - c_{3} & -a_{4} - b_{1} + c_{2}
\end{smallmatrix}\right)
\end{equation}
This gives a clean description of the closed conditions. Let $I$ be the ideal in these variables defined by the $4\times 4$-minors of $\cM$ above, and set
\begin{equation}
  \label{eq:Aprimedef}
  A' \df \Spec(K[a_1,a_2,a_3,a_4,b_1,b_2,b_3,b_4,c_1,c_2,c_3,c_4]/I).
\end{equation}

To understand the open conditions which imply that a sub-algebra gives a quaternion algebra, recall that the trace-$0$ part of any quaternion algebra has a non-degenerate bilinear form. Conversely, it is well-known that any two perpendicular elements of $\OO_0$ with non-zero norm generate a quaternion sub-algebra.
It follows that the bilinear form on the three-dimensional subspace is non-degenerate if and only if the subspace corresponds to the trace-$0$ elements of a quaternion algebra. For the affine chart above, the matrix associated to the bilinear form in the standard basis coming from the chart will be denoted by $Q$, and it is given by:
\begin{align}\label{def:Q}
Q= \left(\begin{smallmatrix}
    3+a_1^2+3a_2^2+a_3^2+3a_4^2&a_1b_1+3a_2b_2+a_3b_3+3a_4b_4 &a_1c_1+3a_2c_2+a_3c_3+3a_4c_4\\
    a_1b_1+3a_2b_2+a_3b_3+3a_4b_4&1+b_1^2+3b_2^2+b_3^2+3b_4^2&b_1c_1+3b_2c_2+b_3c_3+3b_4c_4\\
    a_1c_1+3a_2c_2+a_3c_3+3a_4c_4&b_1c_1+3b_2c_2+b_3c_3+3b_4c_4&3+c_1^2+3c_2^2+c_3^2+3c_4^2
  \end{smallmatrix}
\right).
\end{align}
The open condition is thus $\det(Q) \neq 0$. Therefore, we also introduce the following open subset of $A'$:
\begin{equation}
  \label{eq:Adef}
  A \df \Spec(K[a_1,a_2,a_3,a_4,b_1,b_2,b_3,b_4,c_1,c_2,c_3,c_4,1/\det(Q)]/I)
\end{equation}
Our aim now is to find a slice inside of $A$ that will give a generically finite cover of a large part of our moduli space $X$.

\begin{rmk}
  While we will not need this fact, the determinant $\det(Q)$ can be related to the invariants $r$, $s$ and $t$ introduced in Section \ref{ss:invariants} below via the identity
  \[
  \det(Q) = 9-3t+r-3s.
  \]
\end{rmk}

\section{An open patch in the moduli space}
\label{p: LargePatch}

We now wish to describe the component $X= \GL_2^{\short}\backslash G_2 / \SO_\HH$, or at least a large Zariski dense open subset of it. Our approach uses the following commutative diagram that we will construct in stages below:
\[
  \xymatrix{
    &W\ar[d]_{\subseteq}&&\\
    \widetilde{X}_0 \ar[r]^{g.f.} & X_0 \ar[r]^{g.f. }\ar[d]_{\subseteq}& Z \ar[r]^{g.f}\ar[dd]^{\subseteq} & Y\\
 & A\ar[d]_{\subseteq}\ar[ur]& &   \\
 & G_2/\SO_{\HH} \ar[r]  & X}
\]
The superscript ``g.f.'' above indicates generically finite maps, and the other objects above are studied in the following places below:
\begin{itemize}
\item $A$ is as defined in \eqref{eq:Adef},
\item $Y$ is defined in section \ref{ss:open},
\item $Z = \GL_2^{\short}\backslash A$ will be seen to be a Zariski open subset of $X$ in Theorem \ref{thm:genericfiniteall},
\item both $X_0$ and $\widetilde{X}_0$ are defined in section \ref{s:slice}, and
 \item The composed map $A\to Z \to Y$ is what we call $f$ in section \ref{ss:open}.
 \end{itemize}

 In Section \ref{s:representation} and below we will work mainly with $\widetilde{X}_0$, and use it to construct a family of representations $\Gamma \to G_2(K)$ giving a finite cover of an open part of the character variety of all such representations.

\subsection{Invariant theory for the double coset}
\label{ss:invariants}

Looking at the standard affine chart for the $(3,7)$-Grassmanian described in Section \ref{ss:connected}, we use introduce coordinates as follows:
\begin{equation}
\label{eq:Mdef}
  M=  \left(\begin{smallmatrix}
a_1 & b_1 & c_1 \\
a_2 & b_2 & c_2 \\
a_3 & b_3 & c_3 \\
a_4 & b_4 & c_4
\end{smallmatrix}\right).
\end{equation}
That is, below we will regard $M$ as a generic point of the standard affine chart of the $(3,7)$-Grassmanian.

The action of $\GL_2^{\short} \cong  \HH_1\times \U_1/ \{\pm 1\}$, as described in Section \ref{sss:d2}, can be interpreted in the following way.
Let $ (a+bi+cj+dij ,e+fi) \in  \HH_1\times \U_1/ \{\pm 1\}$, and denote 
\begin{align*}
  h \df&  \left(\begin{smallmatrix} 
a &3b &c &3d \\
-b & a &-d& c\\
-c& 3d &a &-3b\\
-d& -c&  b&a 
\end{smallmatrix}\right)
 \left(\begin{smallmatrix}
e &-3f & 0&0 \\
f &e & 0&  0\\
0 & 0 &e & -3f\\
0 & 0 &f &e 
\end{smallmatrix}\right), & \phi_h \df& \left(\begin{smallmatrix} 
a^2+3b^2-c^2-3d^2 &-2ad+2bc& 2ac+6bd \\
6ad+6bc &a^2-3b^2+c^2-3d^2& -6ab+6cd \\
-2ac+6bd & 2ab+2cd &a^2-3b^2-c^2+3d^2 
\end{smallmatrix}\right).
\end{align*}
  Then the action of $\GL_2^{\short}$   is given by 
  \[ 
(a+bi+cj+dij ,e+fi) \cdot M = h M    \phi_h^{-1}.
 \]

 To begin our investigation of invariants, we first ignore the closed and open conditions that describe $X$, and work initially on the whole linear space of points $M$ from \eqref{eq:Mdef}. The action of $\SL_2^{\short}$ on this space decomposes as copies of the standard represetantation as well as with two copies of the standard representation tensored with the symmetric square. From this it follows that the torus $T^{\short}\subseteq \SL_2^{\short}$ acts with weights
\[
(3,3,1,1,1,1,-1,-1,-1,-1,-3,-3).
\]
Thus, in the notation of part (d) of Section 4.4 in \cite{mukai}, we have that the $q$-Hilbert series of the action on points $M$ is:
\[
  H_s(q;t) = \frac{1}{(1-q^3t)^2(1-qt)^4(1-q^{-1}t)^4(1-q^{-3}t)^2}
\]
Hence, by Proposition 4.6.3 of \cite{mukai}, the Hilbert series of the $\SL_2^{\short}$-invariant ring is
\[
  H_s(t) = -\Res_{q=0}\left((q-q^{-1})H_s(q;t)\right).
\]
A computer calculation reveals that:
\[
  H_s(t) = 1 +2t^2 + 29t^4 + 95t^6 + 390t^8+1056t^{10}+2882t^{12}+6525t^{14}+\cdots
\]
The degree $d$ coefficients here encode the dimensions of the degree $d$ invariants of $\SL_2^{\short}$, which are thus $\GL_2^{\short}$ semi-invariants. A similar computation that takes the determinant into consideration yields the Hilbert series for the $\GL_2^{\short}$-invariants:
\[
  H(t) = 1 + 2 t^{2} + 11 t^{4} + 31 t^{6} + 94 t^{8} + 222 t^{10} + 516 t^{12} + 1047 t^{14} \cdots
\]
We emphasize that these preceding computations ignore some algebraic conditions encoding that our representations take image in $G_2$, and so the invariants of interest below will be fewer in number than what is described above by $H(t)$. This series is primarily useful for us as it provides upper-bounds on the spaces of invariants of interest below.

We are looking to describe a four-dimensional variety by using sufficiently many invariants that they yield an embedding into affine space. Therefore, we already see from $H(t)$ that there are simply not enough invariants of degree two. It turns out that by using invariants of degree $2$ and $4$ we can achieve our goal. A first obstacle in executing this strategy is to compute a basis for these spaces of invariants. We initially used a mixture of methods to do this, and then we massaged our answers into a particularly nice looking basis that led to a reasonably simple description of $X$. We will briefly summarize the three approaches we used.

\subsubsection{Linear algebra and representation theory} One can use the well-known representation theory of $\GL_2$ and linear algebraic constructions to produce invariants for our group by decomposing the linear space of points $M$ into irreducible subrepresentations of $\GL_2$. One issue with this approach is that decomposing the representation leads to invariants with many monomials in them, as well as some complicated coefficients. Therefore we will not say too much about this standard approach other than to summarize one construction that yields three important invariants.

To describe these invariants, let
\begin{align*}
  I_3&= \left(\begin{smallmatrix}
1 &  &  \\ 
&3& \\
&&1
\end{smallmatrix}\right), &I_4&= \left(\begin{smallmatrix}
1 & & &\\
 &3 & &  \\
 &  &1 & \\
&  &&3 
\end{smallmatrix}\right).
\end{align*}
Since $I_3  \left(\, ^{T}\phi_h^{-1}\right) I_3^{-1} = \phi_h $ and $^{T}h I_4 h = I_4$, the map 
\begin{equation}\label{def:Q0} 
M \mapsto  Q_0 = I_3  \,  {^T}M I_4 M
\end{equation} is equivariant for the action of $\GL_2^{\short}$, which on the right is  given  by
\[   (a+bi+cj+dij ,e+fi) \cdot Q_0 =  \phi_h Q_0  \phi_h^{-1}.
\]
The characteristic polynomial of $Q_0$ is given by $x^3+tx^2+rx+9s$, where $t$, $r$ and $s$ are homogeneous of degree $2$, $4$ and $6$, respectively. If we take the closed conditions from Subsection \ref{ss:G2modSOH} underlying $X$ into consideration, then these invariants satisfy the nontrivial, but easily verified, identity
\[
  r= 6s-\tfrac{1}{4}(t-s)^2.
\]
Without the closed conditions, the invariants $t,s,r$ appear to be algebraically independent.


\subsubsection{Computational approach}
\label{s:computational}
A second approach to computing invariants of low degree is simply to intersect $1$-eigenspaces of sufficiently many matrices acting on polynomial functions on $M$ of the given degree. This is computationally feasible at least up to degree $6$, and we are able to compute bases of invariants in degrees $2$ and $4$. The degree $2$ invariants are spanned by 
\begin{align*}
t&= -a_1^2-3a_2^2-a_3^2-3a_4^2-3b_1^2-9b_2^2-3b_3^2-9b_4^2-c_1^2-3c_2^2-c_3^2-3c_4^2,\\
u &= 3\det\stwomat{a_1}{b_1}{a_4}{b_4} - 3\det\stwomat{a_2}{b_2}{a_3}{b_3}-\det\stwomat{a_1}{c_1}{a_3}{c_3}+3\det\stwomat{b_1}{c_1}{b_2}{c_2}-3\det\stwomat{a_2}{c_2}{a_4}{c_4}-3\det\stwomat{b_3}{c_3}{b_4}{c_4},
\end{align*}
where $t$ is the degree $2$ invariant produced in the preceding section.

\subsubsection{Determinantal approach}
After computing bases of invariants using the preceding approaches, we observed that many of the invariants can be described as determinants. Here is a systematic way to construct directly a number of such determinantal invariants.

Let $\cV$ denote $192$-dimensional vector space consisting of the set of $4\times 4$-matrices whose entries are linear functions in our variables $a_1,\ldots, c_4$. If we write elements of $\cV$ as matrix-valued functions $P(M)$ where $M$ is our generic matrix of variables, then we have a right action of $\GL_2^{\short}$
\[
  P(M)\bullet h = h^{-1}P(hM\phi_h^{-1})\stwomat{\phi_h}{0}{0}{1}
\]
Notice that if $P(M)$ is invariant for this action, then $\det P(M)$ is a degree-$4$ invariant for the action $h \cdot M = hM\phi_h^{-1}$ on $M$.

The space of invariants in $\cV$ is $6$-dimensional, spanned by:
\begin{align*}
P_1 &= \left(\begin{smallmatrix}
a_{1} & b_{1} & c_{1} & 0 \\
a_{2} & b_{2} & c_{2} & 0 \\
a_{3} & b_{3} & c_{3} & 0 \\
a_{4} & b_{4} & c_{4} & 0
\end{smallmatrix}\right), & P_2 &= \left(\begin{smallmatrix}
0 & 0 & 0 & a_{1} - 3 b_{4} + c_{3} \\
0 & 0 & 0 & a_{2} + b_{3} + c_{4} \\
0 & 0 & 0 & -a_{3} + 3 b_{2} + c_{1} \\
0 & 0 & 0 & -a_{4} - b_{1} + c_{2}
\end{smallmatrix}\right),\\
  P_3&= \left(\begin{smallmatrix}
-3 a_{2} & -3 b_{2} & -3 c_{2} & 0 \\
a_{1} & b_{1} & c_{1} & 0 \\
-3 a_{4} & -3 b_{4} & -3 c_{4} & 0 \\
a_{3} & b_{3} & c_{3} & 0
\end{smallmatrix}\right),&
  P_4 &= \left(\begin{smallmatrix}
0 & 0 & 0 & -3 a_{2} - 3 b_{3} - 3 c_{4} \\
0 & 0 & 0 & a_{1} - 3 b_{4} + c_{3} \\
0 & 0 & 0 & 3 a_{4} + 3 b_{1} - 3 c_{2} \\
0 & 0 & 0 & -a_{3} + 3 b_{2} + c_{1}
\end{smallmatrix}\right), \\ P_5 &=\left(\begin{smallmatrix}
-3 b_{3} - 3 c_{4} & a_{3} - c_{1} & 3 a_{4} + 3 b_{1} & 0 \\
-3 b_{4} + c_{3} & a_{4} - c_{2} & -a_{3} + 3 b_{2} & 0 \\
-3 b_{1} + 3 c_{2} & a_{1} + c_{3} & -3 a_{2} - 3 b_{3} & 0 \\
-3 b_{2} - c_{1} & a_{2} + c_{4} & a_{1} - 3 b_{4} & 0
\end{smallmatrix}\right), & P_6&=\left(\begin{smallmatrix}
-3 b_{4} + c_{3} & a_{4} - c_{2} & -a_{3} + 3 b_{2} & 0 \\
b_{3} + c_{4} & -\frac{1}{3} a_{3} + \frac{1}{3} c_{1} & -a_{4} - b_{1} & 0 \\
-3 b_{2} - c_{1} & a_{2} + c_{4} & a_{1} - 3 b_{4} & 0 \\
b_{1} - c_{2} & -\frac{1}{3} a_{1} - \frac{1}{3} c_{3} & a_{2} + b_{3} & 0
\end{smallmatrix}\right).
\end{align*}
Taking determinants of various linear combinations of these matrices yields the following $8$ linearly independent degree $4$ invariants:
\begin{align*}
  D_1 &=\det(P_1+P_2), & D_2 &=\det(P_1+P_4),\\
  D_3 &=\det(P_4+P_5), & D_4 &=\det(P_2+P_5),\\
  D_5 &=\det(P_1+P_4+P_5), & D_6 &=\det(P_1+P_4+P_6), \\
  D_7 &=\det(P_1+P_2+P_5+P_6), & D_8 &=\det(P_1+P_2-P_5-P_6).
\end{align*}
These can then be augmented to a basis for the space of degree $4$ invariants by adding
\begin{align*}
D_{9}&= t^2,\\
  D_{10}&=u^2,\\
  D_{11} &= \left(\det\stwomat{b_1}{a_1}{b_2}{a_2}+\det\stwomat{b_3}{a_3}{b_4}{a_4}\right)^2+\tfrac{1}{3}\left(\det\stwomat{c_1}{a_1}{c_2}{a_2}+\det\stwomat{c_3}{a_3}{c_4}{a_4}\right)^2+\left(\det\stwomat{c_1}{b_1}{c_2}{b_2}+\det\stwomat{c_3}{b_3}{c_4}{b_4}\right)^2.
\end{align*}
Again, up to now we have ignored the closed conditions underlying $X$. We now address these conditions.

\subsubsection{Addressing the closed conditions} 
These preceding invariants have been ordered such that $D_1$, $D_2$ and $D_3$ are exactly those invariants that are contained in the ideal defined by our closed conditions, and so effectively we can ignore them. More generally, one computes that modulo the closed conditions defining $X$ --- which are not homogeneous --- the following relations hold:
\begin{align*}
  -(1/3)D_4+(1/12)D_9-(1/4)D_6-(1/3)D_{10} &=t+2u,\\
    (4D_4-D_9+3D_6+4D_{10})^2 &= 432D_6,\\
    t^2+2tu-D_5+(3/2)(D_7+D_8)&=10t+20u.
\end{align*}
Thus, after imposing the closed conditions we may also ignore the invariants $D_5$ and $D_6$. Since we are using $t$ and $u$ in our embedding data, we can also omit $D_9$ and $D_{10}$. A computation shows that these invariants satisfy $P(t,u,D_7,D_8,D_{11})=0$ where
\begin{align*}
  & P(t,u,D_7,D_8,D_{11})\\
  =& 18D_7+18(3+D_{11})t+3(36+D_7-D_8+12D_{11})u+96u^2+48tu+8tu^2+16u^3.
\end{align*}

Thus, let
\begin{align*}
  Y' &= \Spec(K[t,u,D_7,D_8,D_{11}]/(P)) \\
  &= \{(t,u,D_7,D_8,D_{11}) \in \bA^5\mid P(t,u,D_7,D_8,D_{11}) = 0\}.
\end{align*}
Then our invariants $t,u,D_7,D_8,D_{11}$ can be used to define a map $f\colon A' \to Y'$, where $A'$ is as in \eqref{eq:Aprimedef}, that descends to the quotient of $A'$ by $\GL_2^{\short}$, since the map is defined by invariants for this action. 

\subsubsection{Addressing the open condition}
\label{ss:open}
We next describe how the open condition interacts with this map $f$. For this, let $\det(Q)$, where $Q$ is as in \eqref{def:Q}, be the invariant corresponding to the nonzero discriminant condition. Note that in terms of previous invariants we have
\[
  \det(Q) = 9-3t+r-3s
\]
and this is an equality without needing the closed conditions. If we take the closed conditions into consideration then we have the following identities:
\begin{align*}
  s &= t-2u,\\
  r &= u^2+6t+12u,\\
  \det(Q) &= (u+3)^2.
\end{align*}
Therefore, the map we wish to study, taking both the closed and open conditions into consideration, has the following description:
\begin{align*}
  A &= \Spec(\QQ[a_1,a_2,a_3,a_4,b_1,b_2,b_3,b_4,c_1,c_2,c_3,c_4,1/Q]/I),\\
  Y &= \{(t,u,D_7,D_8,D_{11}) \in \bA^5\mid P(t,u,D_7,D_8,D_{11}) = 0,~ (u+3)^2\neq 0\},\\
  g &\colon A \to Y,\\
  g(M) &= (t,u,D_7,D_8,D_{11}).
\end{align*}
Recall that $I$ is the ideal of closed conditions generated by the $4\times 4$-minors of the matrix $\cM$ from \eqref{eq:cM}. The map $g$ factors through $X$, and we wish to find a subvariety $S \subseteq A$ such that $f\colon S \to X$ is finite with large image.

\subsection{A slice in $A$ }
\label{s:slice}
The aim of this section is to identify a subvariety $X_0$ of $A$ for which the map $X_0 \rightarrow X$ is generically finite. To this end we wish to identify particularly simple representatives for elements of $X$. Recall that the association $M \mapsto Q_0$ as in \eqref{def:Q0} gives an $\SO_{\HH}$-equivariant map from $A$ to $3 \times 3$-symmetric matrices where the action of $(h_1,h_2)\in \SO_{\HH}$ on $Q_0$ is via $h_2$ acting as $\phi_{h_2} Q_0 \phi_{h_2}^{-1}$. Consequently, we have the following Lemma:

\begin{lem}
  \label{lem:modSO4}
The following hold:
\begin{enumerate}
 \item every element of $A$ has a $\GL_2^{\rm short}$-representative where $Q_0$ is diagonal.
 \item two diagonal matrices $Q_0$ are in the same $\GL_2^{\rm short}$-orbit if they have the same diagonal entries up to permutation.
\item every element of $A$ where $Q_0 \neq 3\Id_3$  has an $\SO_{\HH}$ representative of the form 
\[
 \begin{pmatrix} 0 & 0 & 0 \\
                           a_2 & 0 & 0 \\
                           0 & b_3 & 0 \\
                           0 & 0 & c_4 \end{pmatrix} 
                           \qquad \text{where}\qquad
  a_2b_3c_4-a_2-b_3-c_4=0.
\]
We denote by $W$ the subvariety $A$ consisting of such elements. 
\end{enumerate}

\end{lem}
\begin{proof}
The first claim follows from an appropriately general version of the spectral theorem as the elements $\phi_h$ surject onto the special orthogonal group for the bilinear form on ${\HH}_0$. Notice then that if we restrict to considering matrices $M$ mapping onto diagonal matrices $Q_0$, the columns of $M$ are then orthogonal. We will make use of this observation when proving claim (3) below.

The second claim can be proved explicitly, consider for example the elements
\[  \tfrac{1}{\sqrt{2}}+\tfrac{1}{\sqrt{6}}i, \quad \tfrac{1}{\sqrt{2}}+\tfrac{1}{\sqrt{2}}j, \quad  \tfrac{1}{\sqrt{2}}+\tfrac{1}{\sqrt{6}}ij \]
in ${\HH}_1$. These act like the three transpositions of the diagonal entries of $Q_0$ and thus they generate all permutations of the diagonal entries.

For the third claim notice that if $Q_0  \neq 3{\rm Id}_3$ is non-zero, then without loss of generality the diagonal entry of $Q_0$ which is not $3$ can be ordered to be the first. Thus, if we define $\vec{a} = a_1 + a_2i + a_3j+a_4ij$ using the first column of a matrix $M \in A$ mapping onto $Q_0$, then $\vec{a}$ is a non-degenerate element of ${\HH}$, and hence there is $h_1\in {\HH}_1$ so that the only nonzero coordinate of $h_1\vec{a} = (0,a_2,0,0)$ is the $i$-coordinate. Therefore, using our orthogonality observation in the proof of claim (1), every element of $A$ has a representative of the form
\[
 \begin{pmatrix} 0 & b_1 & c_1 \\
                           a_2 & 0 & 0 \\
                           0 & b_3 & c_3 \\
                           0 & b_4 & c_4 \end{pmatrix} \]
where the second two columns are perpendicular, thought of as elemenets of ${\HH}$. Having specialized $a_1$, $a_3$, $a_4$, $b_2$ and $c_2$ to be zero, the condition from \eqref{augmat1} that the fourth column of the matrix in \eqref{augmat1} should be in the span of the first three quickly gives 
\[ (a_2^2-1)b_1 = 0. \]
The open condition, $\det Q_0\neq 0$, implies $a_2^2-1 \neq 0$ giving $b_1=0$. Similarly, using the fifth column in \eqref{augmat1} being in the span of the first three gives $c_1=0$. Now, looking at the fourth row of the sixth column gives $c_3=3b_4$. Perpendicularity between the second and third column then gives $b_4(b_3+c_4) = 0$.

In the case $b_4=0$ we have the desired shape, so consider the case $c_4=-b_3$.
Looking at the first four columns now gives
\[ -a_2-b_3 = 3a_2b_4 b_2 - (1-a_2b_3)b_3 \]
which gives $a_2(1+b_3^2+3b_4^3)$, however, by assumption $a_2\neq0$ and $1+b_3^2+3b_4^3\neq 0$ from the non-degeneracy of $Q_0$. 

Now, given that the only non-zero entries are $a_2$, $b_3$, and $c_4$
one may then directly verify from the third and forth column of \eqref{augmat1} the equation
\[
  a_2b_3c_4-a_2-b_3-c_4=0.\qedhere
\]
\end{proof}

\begin{rmk}
Note that, in the preceding proof, when $Q_0$ is zero, the associated representation will factor through an $\SU$-subgroup. In particular, the special cases when $a_2=0$ all factor through an $\SU$-subgroup. See Section \ref{ss:factoringA2} for details.
\end{rmk}

\begin{lem}
\label{lem: gl2reps}
Every element of $A$  where $Q_0 \neq 3\Id_3$ has a $\GL_2^{\rm short}$-representative of the form:
\[
 \begin{pmatrix} -3ax & -by & -3cz \\
                           aw & -bz & cy \\
                           3az & bw & -3cx \\
                           -ay & bx & cw \end{pmatrix} 
                           \qquad \text{where}\qquad
   abc-a-b-c=0 \text{ and } w^3+3x^2+y^2+3z^2=1.
\]
Let $X_0$ denote the subset of $A$ where the elements all have a representative of the form:
\[
 \begin{pmatrix} -3ax & -by & 0 \\
                           aw & 0 & cy \\
                           0 & bw & -3cx \\
                           -ay & bx & cw \end{pmatrix} 
                           \qquad \text{where}\qquad
   abc-a-b-c=0 \text{ and } w^3+3x^2+y^2=1.
\]

Then the induced map $X_0 \rightarrow X$ has open image and is generically finite.
 \end{lem}
 \begin{proof}
 For the first claim we simply notice that this is precisely the set of ${\HH}_1 \times 1$ orbits of elements of the form
 \[
 \begin{pmatrix} 0 & 0 & 0 \\
                           a & 0 & 0 \\
                           0 & b & 0 \\
                           0 & 0 & c \end{pmatrix} 
                           \qquad \text{where}\qquad
  abc-a-b-c=0.
\]
Therefore the first claim follows by Lemma \ref{lem:modSO4}.

The second claim follows from the observation that the map
\[ \{ (w,x,y)\;|\; w^2+3x^2+y^2\} \rightarrow {\HH}_1 /\U_1 \]
has open image and is generically finite, and that
\[
 \begin{pmatrix} -3ax & -by & 0 \\
                           aw & 0 & cy \\
                           0 & bw & -3cx \\
                           -ay & bx & cw\end{pmatrix} 
                           \qquad \text{where}\qquad
  abc-a-b-c=0 \text{ and } w^2+3x^2+y^2=1 
\]
is the natural image of $\{ (w,x,y)\;|\; w^2+3x^2+y^2\} $, which is not a subgroup, acting on 
 \[
 \begin{pmatrix} 0 & 0 & 0 \\
                           a & 0 & 0 \\
                           0 & b & 0 \\
                           0 & 0 & c \end{pmatrix} 
                           \qquad \text{where}\qquad
  abc-a-b-c=0.
\]
Because $X_0$ is $4$-dimensional and surjects onto an open in $X$, which is also $4$ dimensional, we conclude that the map is generically finite.
 \end{proof}

\begin{rmk}
We note that one could cover all of ${\HH}_1 /\U_1$ by a finite collection of similar sets to find all elements where $Q_0$ is not equal to $3\Id_3$. Those elements which do not have a representative as above will all be of the form $(w,x,y,z)$ where $w^2+3x^2=1$ and $y^2+3z^2=0$. As these comprise a subvariety of dimension at most $3$ in our ambient $4$-dimensional space, we will ignore these nongeneric families below in our analysis, for the sake of brevity.
\end{rmk}

\begin{lem}\label{lem:X0}
  Let $\widetilde{X}_0$ denote the space
  \[
 \widetilde{X}_0 = \{  ((w,x,y),(a,b,c))\;|\; w^2+3x^2+y^2=1,\;abc-a-b-c=0\}.
  \]
  Then the map
\[ \widetilde{X}_0 \rightarrow X_0 \subset A \]
given by
\[  ((w,x,y),(a,b,c)) \mapsto  \begin{pmatrix} -3ax & -by & 0 \\
                           aw & 0 & cy \\
                           0 & bw & -3cx \\
                           -ay & bx & cw\end{pmatrix}   \]
is $2$ to $1$.
\end{lem}
\begin{proof}
Clearly $((w,x,y),(a,b,c))$ and $((-w,-x,-y),(-a,-b,-c))$ map to the same point, and it is not hard to verify that these are the only identifications that take place under this map.
\end{proof}

\begin{lem}
  For $m\in W$ denote by $H_m = \{ h\in \SO_{\HH}\;|\; hm \in W \}$. Then for generic choices of $m$, the set $H_m$ is a subgroup that can be described explicitly as follows: set
  \begin{align*}
    H'& =  \left\langle \tfrac{1}{\sqrt{2}} + i\tfrac{1}{\sqrt{6}}, i\tfrac{1}{\sqrt{6}}+j\tfrac{1}{\sqrt{2}},-1\right\rangle,\\
    H''& =  \left\langle i\tfrac{1}{\sqrt{3}}, j,-1\right\rangle,
  \end{align*}
  which are groups of order $48$ and $8$, respectively. Then
  \[
H_m = \{(h_1,\pm h_1) \in \SO_{\HH} \mid h_1 \in H'\}
  \]
  is a group of order $48$ inside $\SO_{\HH} = \SL_2^2/\{\pm 1\}$. Moreover, the stabilizer in $H_m$ of $m$ is
  \[
\{(h_1,h_1) \in \SO_{\HH} \mid h_1 \in H''\},
\]
which is a subgroup of order $4$.
\end{lem}
\begin{proof}  
  Generically we have that the diagonal elements of $Q_0$ are distinct, and none of them are $3$. It follows that any $h\in H_m$ acts to permute the diagonal of $Q_0$, and, consequently, permute the values $a$, $b$, and $c$ up to sign.
 
 The subgroup of $\SO_{\HH}$ which permutes the diagonal of $Q_0$ is given by $H'\times \HH_1$ where 
 \[  H' =  \left\langle \tfrac{1}{\sqrt{2}} + i\tfrac{1}{\sqrt{6}}, i\tfrac{1}{\sqrt{6}}+j\tfrac{1}{\sqrt{2}},-1\right\rangle.\]
This is a group of size $48$ isomorphic to $S_4\times C_2$, where the leftmost generator listed above represents a $4$-cycle, and the middle generator represents an element of order $2$. Mapping elements $h \in H'$ to their representing matrix $\phi_h$ defines a surjective homomorphism $H' \to S_4$ with kernel $C_2 = \{\pm 1\}$, and then this group maps surjectively onto $S_3$ via its (generic) permutation action on the elements of $Q_0$.
 
 Either from the proof of Lemma \ref{lem:modSO4}, or by a direct finite computation since $H'$ is a finite group, one sees for each $h' \in H'$ that there exists an $h''\in \HH_1$ such that  $(h',h'')\in H_m$. In fact, according to our conventions for the various actions, one can take $h''= h'$.  Moreover, it is clear that $h''$ is unique up to sign, noting also that $(-1,-1) = (1,1) \in \HH_1 \times \HH_1/(\pm 1) = \SO_{\HH}$. This concludes the proof of our description for $H_m$, and then one sees from this description that $H_m$ is in fact a subgroup of $\SO_{\HH}$.

Finally, the claim about the stabilizer of $m$ in $H_m$ can simply be checked by enumerating the elements of $H'$ and checking.
\end{proof}

\begin{rmk}\label{rmk:generi}
For the purpose of Lemma \ref{lem:X0} the generic condition on $m$ is that $a^2$, $b^2$, and $c^2$ are distinct, which, in consideration of the open condition, is equivalent to $abc\neq 0$ and $a$, $b$, and $c$ being distinct.
\end{rmk}

In order to augment the preceding discussion to understand the fibers of the map $f|_S: S \rightarrow Y$, which will be generically finite, we shall use the following. 

\begin{lem}
Let $(1,h),(1,h') \in \SL_2^{\llong}$ and $m,m' \in W$ with $m$ generic as in Remark \ref{rmk:generi}. Let $(1,h)m$ and $(1,h')m'$ respectively denote the corresponding points in $A$. Suppose that
$(g,1)\in \SL_2^{\rm short}$ satisfies $(g, 1)(1,h) m  = (1,h')m'$. Then
\begin{enumerate}
\item $g\in H'$.
\item $m'=(g,g) m$, and
\item $ h= \pm h'g $.
\end{enumerate}
\end{lem}
\begin{proof}
Firstly, as $(g,1)(1,h)m  = (1,h')m'$ we have 
\[  (1,h'^{-1})(g,1)(1,h) \in H_m. \]
Thus we may write 
\[   ({h}_1,\pm {h}_1) =  (1,h'^{-1})(g,1)(1,h) \]
where $h_1\in H'$.

It follows that $g=h_1 \in H'$. Hence, $ (g,\pm g) = (g, {h'}^{-1}h).$
so that $h = \pm h' g.$
\end{proof}

\begin{rmk}
The points $(1,h)m$ and $(1,h')m'$ need not be in $X_0$ because $h$ and $h'$ may include the variable $z$.
\end{rmk}

\begin{lem}\label{lem:Ghm}
  For $(1,h)m\in X_0$ with $h=w+xi+yj\in  \SL_2^{\llong}$ and $m\in W$ generic as in Remark \ref{rmk:generi}, let
  \[G_{(1,h)m} = \{ (g,u)\in \GL_2^{\rm short} \;|\; (g,u)(1,h)m \in X_0 \}.\]
    Then the following hold:
\begin{enumerate}
\item for $(g,u) \in G_{(1,h)m}$ we have $g \in H'$;
\item $u$ is determined, up to multiplication by $\pm 1$, by $h$ and $g$, as in the proof below.  
\end{enumerate}
In particular, for generic choices of $m$, the set $G_{(1,h)m}$ is of size $48$.

Moreover, for generic choices of $(1,h)$ the stabilizer of $(1,h)m$ in $G_{(1,h)m}$ is trivial.
\end{lem}
\begin{proof}
For $(g,u)\in G_{(1,h)m}$  write $(g,u) (1,h)m = (1,h')m'$ with $h' = w'+x'i+y'j\in \SL_2^{\llong}$.

Now we see that 
\[ (1,(h')^{-1})(g,u) (1,h)m = (g,(h')^{-1}uh)= m'. \]
It follows that $ (g,(h')^{-1}uh) \in H_m$ and hence $g\in H'$ and
\[ g = \pm(h')^{-1}uh \]
so that
\[ \pm(h')^{-1} u = g h^{-1}. \]
Write $ g h^{-1} =  w_1 + x_1i + y_1j + z_1ij$ and $u=u_1+u_2i$ and we see
\[ \pm y'j (u_1+u_2i) =  y_1j + z_1ij \]
so that generically we have
$(u_1,u_2) = \pm (y_1,z_1)/\sqrt{y_1^2+3z_1^2}$.

Now, if in the above we had
\[ (g,u)(1,h)m = (1,h)m \]
then 
\[ (g,h^{-1} u h) m = m \]
hence $g \in H''$ and $u= hgh^{-1} \in H''$. For generic $h$ the condition $u\in \U_1$ implies $g=\pm 1$, but this gives $(g,u) = \pm(1,1)$ which is trivial in $\SO_\HH$.
%
%
%
%
\end{proof}

\begin{rmk}\label{rmk:generih}
For the purpose of Lemma \ref{lem:Ghm} the generic condition on $(1,h)$ is that it does not conjugate elements of $H''$ into $\U_1$.
We see that this is equivalent to $wxy\neq 0$, indeed, if in the above we had
\[ (g,u)(1,h)m = (1,h')m \]
then 
\[ (g,(h')^{-1} u h) m = m \]
hence $g = (h')^{-1} u h \in H''$ and $uh= h'g$.
Given $h' = w'+x'i+y'j$ and $h=w+xi+yj$ and by checking cases for $g\in H''$ we see that $u=\pm g$ and hence.
 $h'=\pm gh'g^{-1}$
It follows that we are specifically identify all of 
\[ \pm (w+xi+yj)m, \pm (w-xi+yj)m, \pm (w+xi-yj)m, \pm (w-xi-yj)m. \]
By the orbit stabilizer theorem we obtain a larger stabilizer if and only if $wxy=0$.
\end{rmk}

\begin{lem}\label{lem:genericfinitepart}
  The composite map
  \[
    W \rightarrow ({\HH}_1\times{\HH}_1) \backslash A \to \bA^2
  \]
 given by the degree $2$-invariants $(t,u)$ from Section \ref{ss:invariants} is a finite map of degree $12$ with Zariski dense image. The second map in this composition is an isomorhpism on a Zariski dense subset.
\end{lem}
\begin{proof}
  We can identify $W$ with the affine scheme
  \[
W = \{(a,b,c)\;|\; abc-a-b-c=0,~ u\neq -3\}.
\]

On this subscheme we have that
\begin{align*}
  t(a,b,c) &= -3(a^2+b^2+c^2),\\
  u(a,b,c) &= -3(ab+ac+bc).
\end{align*}
That is, this map is given by symmetric polynomials. Sage outputs that the scheme-theoretic image of this map is the entire affine plane. But I think this may simply be returning the closure of the set-theoretic image. For example, if we include the open condition $u\neq -3$ then obviously we lose surjectivity and the image is the basic affine open described by this condition. But Sage still says this is surjective. So probably as long as we don't care what the Zariski dense set-theoretic image is, only that it is Zariski dense, this Sage computation is enough.

Consider the ideal in $\QQ[a,b,c,t,u]$ defined as
\[I=(abc-a-b-c,t+3(a^2+b^2+c^2),u+3(ab+ac+bc)).\]
Eliminating $a$ and $b$ from this ideal produces the principal ideal defined by
\[
9c^6+3tc^4+(u^2+6t+12u)c^2+3t+6u \quad \iff \quad t = -\frac{9c^6+u(u+12)c^2+6u}{3(c^2+1)^2}
\]
Therefore, given each pair $(t,u) \in \bA^2$, we find that there are generically $6$ possible choices for $c$ in the preimage. There are fewer than $6$ choices for $c$ precisely if
\[
(t+2u)(u+3)(-tu^2+2u^3-12t^2+6tu+60u^2+243t+486u)=0.
\]
Since our equations are symmetric in $a$, $b$ and $c$, the same holds if we eliminate any other pair of variables.

Suppose that we avoid the degenerate locus of points $(t,u)$ described by the vanishing above, and choose a triple $(a,b,c)$ in the preimage. We have seen that there are $6$ possibilities for $c$. If $bc=1$ then the closed condition simplifies to $b+c=0$ and we find that $u=-3$. Therefore, this does not occur for any of the points $(a,b,c)$ above $(t,u)$ in the nondegenerate locus. It follows that we can use the closed condition to solve for $a=\frac{b+c}{bc-1}$ above the nondegenerate locus. One observes that then
\begin{align*}
  t &= -3\frac{c^2b^4-2cb^3+(c^4+2)b^2+2(1-c^2)cb+2c^2}{(bc-1)^2},\\
  u &= -3\frac{(c^2+1)b^2+cb+c^2}{bc-1}.
\end{align*}
For each of the $6$ different $c$ values that can lie above $(t,u)$, if $c^2+1\neq 0$ then the expression for $u$ above shows that there are generically $2$ choices for $b$. On the other hand, if $c^2+1=0$, in fact $u=-3$ so we are in the degenate locus and we can ignore this case. It follows that the fibers of the map under study are generically of size at most $12$.

Now we should show that the maximum size of $12$ is actually obtained generically. Therefore, given $(u,t)$ in the nondegenerate locus, choose a $c$ value satisfying the degree $6$ polynomial above with coefficients in $\QQ[u,t]$. We then solve for $b$ as a root of the quadratic polynomial
\[
(3c^2+3)b^2+(cu+3c)b+3c^2-u,
\]
so that
\[
b= \frac{-cu-3c \pm \sqrt{-36c^4+c^2u^2+18c^2u-27c^2+12u}}{6(c^2+1)}
\]
This allows us to express $a$ and $b$ in terms of $c$ and $u$ alone. We have seen above that since $c^2+1 \neq 0$ in the nondegnerate locus, we can also solve for $t$ in terms of $c$ and $u$ alone. Doing so, one can check say on a computer that $-3(a^2+b^2+c^2)$ is indeed equal to the expression for $t$ in terms of $c$ and $u$. Therefore, each of the six choices of $c$ leads, generically, to two points $(a,b,c)$ above $(u,t)$, so that fiber size is indeed generically equal to $12$.
\end{proof}

\begin{rmk}
  \label{rmk:nondegen}
  In terms of the parameters $a,b,c$ we have $t+2u = (a+b+c)^2$. Therefore, in the nondegeneracy locus above we have $a+b+c \neq 0$ and, thanks to the closed condition, this is equivalent to assuming that all three of $a$, $b$ and $c$ are nonzero. The interesting cubic factor
  \[-tu^2+2u^3-12t^2+6tu+60u^2+243t+486u\]
  is a singular curve of genus $0$, where the singular point is $(t,u) = (-27,-27)$, which corresponds to $(a,b,c) = (\pm\sqrt{3},\pm\sqrt{3},\pm\sqrt{3})$. Modulo the closed condition and expressed in terms of $a,b,c$, this cubic equation nonvanishing is equivalent to no two of $a,b,c$ being equal. Thus, in the nondegeneracy locus, we can assume all of $a,b,c$ are nonzero and pairwise distinct.
\end{rmk}

\begin{thm}\label{thm:genericfiniteall}
The composite map 
\[ f\colon \widetilde{X}_0 \stackrel{f_1}{\rightarrow} X_0 \stackrel{f_2}{\rightarrow}Z \stackrel{f_3}{\rightarrow} Y \]
given by
\[ ((w,x,y),(a,b,c)) \mapsto  (t,u,D_7,D_8,D_{11}). \]
satisfies:

\begin{enumerate}
\item $f$ has Zariski dense image in $Y$, indeed the image includes the complement of 
\[ t+2u=0 \quad\text{and}\quad -tu^2+2u^3-12t^2+6tu+60u^2+243t+486u = 0. \]
\item $f$ is finite away from:
  \begin{itemize}
  \item $abc(a-b)(a-c)(b-c)=0$,
  \item $y=0$,
  \item $w=x=0$,
  \end{itemize}
  and it is furthermore unramified of degree $192$ if additionally $wx\neq 0$.
  
\item As a consequence of Lemma  \ref{lem:X0}, the map $f_1$ has generic degree $2$. Indeed, it has degree $2$ unless $a=b=c=0$.

\item As a consequence of Lemma \ref{lem:Ghm}, the map $f_2$ has generic degree $48$. 
Indeed it has degree $48$ provided $abc(a-b)(a-c)(b-c)wxy \neq 0$.

\item As a consequence of the above, the map $f_3$ has generic degree $2$.
\end{enumerate}
\end{thm}
\begin{proof}
A direct computation shows that:
\begin{align*}
  t((1,0,0),(a,b,c)) &= -3(a^2+b^2+c^2),\\
  u((1,0,0),(a,b,c)) &= -3(ab+ac+bc),\\
  D_7((1,0,0),(a,b,c)) &= -3(b+c)(a+b+c)(4a^2+5ab+b^2+5ac+5bc+c^2),\\
  D_8((1,0,0),(a,b,c)) &= -3(b+c)(a+b+c)(-4a^2-3ab+b^2-3ac-5bc+c^2),\\
  D_{11}((1,0,0),(a,b,c)) &= c^2b^2,
\end{align*}
and that $t((w,x,y),(a,b,c))$ and $u((w,x,y),(a,b,c))$ are both independent of $w,x,y$. On the other hand, let us write
\[
\delta_j((w,x,y),(a,b,c)) = D_j((w,x,y),(a,b,c))-D_j((1,0,0),(a,b,c))
\]
for $j=7,8,11$. Then we find that if we write $w_0 = w^2$, $x_0 = 3x^2$ and $y_0 = y^2$, so that $w_0 +x_0+y_0=1$, then
\begin{align*}
  \delta_7 &= 12(abc)^3y_0(w_0(c-a)+x_0(b-a)),\\
  \delta_8 &= 12(abc)^2y_0(w_0(c^2-a^2+ab-bc)+x_0(ac-bc+b^2-a^2)),\\
  \delta_{11} &=  -4(abc)y_0(w_0(b^2c^2-a^2b^2)+x_0(b^2c^2-a^2c^2)). 
\end{align*}
One finds that
\[
  (6+u)\delta_7-u\delta_8+6(t+2u)\delta_{11}=0.
\]
In the nondegeneracy locus, Remark \ref{rmk:nondegen} shows that $t+2u\neq 0$. Therefore, in this locus we can use the equation above to solve for $\delta_{11}$ in terms of the other parameters.

Fix now a (generic) triple $(a,b,c)$ satisfying $abc = a+b+c$. We wish to understand the map
\[
(w_0,x_0,y_0) \mapsto 12y_0 \twomat{(abc)^3(c-a)}{(abc)^3(b-a)}{(abc)^2(c^2-a^2+ab-bc)}{(abc)^2(ac-bc+b^2-a^2)}\twovec{w_0}{x_0} = \twovec{\alpha}{\beta},
\]
where $w_0+x_0+y_0=1$. The two-by-two matrix above is generically invertible, as long as none of $a,b,c$ are equal, and they are all nonzero. In this locus of points we can solve for any given $\alpha,\beta$ giving expression for $y_0w_0$ and $y_0x_0$ in terms of $\alpha,\beta,a,b,c$. This equation has two solutions unless $t+2u=0$.

Now let us summarize how the computations above have established each part of the theorem:
\begin{enumerate}
\item The proof above shows that we are surjective onto fibers unless the conditions from (1) hold. In Remark \ref{rmk:nondegen} we observed that these conditions are equivalent to $abc(a-b)(a-c)(b-c) \neq 0$.
\item This follows essentially because the matrix
  \[
\twomat{(abc)^3(c-a)}{(abc)^3(b-a)}{(abc)^2(c^2-a^2+ab-bc)}{(abc)^2(ac-bc+b^2-a^2)}
\]
appearing above is intertible unless $t+2u=0$. Then by studying the expressions for $\delta_7$ and $\delta_8$, we obtain the conditions in part (2). We see that the degree is $192$ if these conditions hold as $12$ is contributed from the first two coordinates by Lemma \ref{lem:genericfinitepart}, $8$ are from the map $(w,x,y) \mapsto (w_0,x_0,y_0)$, and $2$ are contributed by $(w_0,x_0,y_0) \mapsto (\delta_7,\delta_8)$, by the computation above.
\item This follows by Lemma \ref{lem:X0} as stated.
\item This follows by Lemma \ref{lem:Ghm} as stated.
\item This is deduced from the previous parts.
\end{enumerate}
This concludes the proof.
\end{proof}

 \begin{rmk}
    Throughout Section \ref{p: LargePatch}, and in particular in the proof of Theorem \ref{thm:genericfiniteall}, various algebraic conditions have arisen that will reappear in a fundamental way in our discussion in Section \ref{s:maximal} when we study surjectivity of representations.
 \end{rmk}

 \subsection{Other Cells in the Moduli Space}
 \label{ss:othercells}

The particular affine chart of the Grassmanian we studied in the previous section is particularly convenient because it is stable under the action of $\GL_2^{\short}$.
It is also convenient that it allows us to find a large open subset of the moduli space.

In this section we shall briefly discuss other cells in the decomposition of $G_2 / \SO_{\HH}$.
We note that there are several discrete invariants we can attach to elements of $G_2 / \SO_{\HH}$ when we view them as subspaces of $\OO_0$.
The two most obvious are the dimensions of the projections onto $\HH_0\oplus 0 \subset \OO_0$ and  $0\oplus \HH \subset \OO_0$, or equivalently by the rank-nullity theorem, the dimensions of the intersections with  $\HH_0\oplus 0 \subset \OO_0$ and  $0\oplus \HH \subset \OO_0$.
The chart we have already considered includes all points where the projection onto $\HH_0\oplus 0$ has dimension $3$.
 The two other options for the dimension of the projection onto $\HH_0\oplus 0$ are $2$ or $1$ dimensional.

When the dimension of the projection onto $\HH_0\oplus 0$ is $2$ dimensional, then the kernel is a $1$-dimensional subspace of $0 \oplus \HH$. 
We note that $\GL_2^{\short}$ acts on the one dimensional subspaces of $0 \oplus \HH$ and has exactly $4$ orbits. Three isotropic orbits, and, one non-isotropic orbit. It is thus natural to study this portion of the moduli-space as $4$ distinct components and construct a slice of the Grassmanian for each by fixing a representative for each of the $4$ orbits.

For example, for the case of the non-isotropic orbit the line spanned by $k=(0,1)$ can be chosen without loss of generality.
Then a point in $G_2 / \SO_{\HH}$ is determined by a $2$ dimensional non-degenerate subspace $V$ of $k^\perp$ for which the subspace spanned by $1$, $k$, and $V$ is closed under multiplication. Such a two dimensional subspace is uniquely determined by a choice of lines, one from each of the two eigenspaces for the action of $k$ on $k^\perp$, which together with $k$, generate a quaternion algebra.
One then needs only consider the quotien by the stablizier of the line spanned by $k$ in $\GL_2^{\short}$, which is $1$-dimensional, though disconnected. 

For the orbits for isotropic lines one could proceed similarly, the characterization of the complementary quadratic subspaces is more complex and the stablizers in $\GL_2^{\short}$ will be higher dimensional.

When the dimension of the projection onto $\HH_0\oplus 0$ is $1$ dimensional, then the kernel, $V$,  is a $2$ dimensional subspace of $0 \oplus \HH$. 

It this subspace is non-degenerate it is immediate that the choice of such a subspace completely determines a point of $G_2 / \SO_{\HH}$.
More generally, picking any $2$ basis vectors $x$ and $y$ for $V$ we would have that $xy\in \HH \oplus 0$, but $xy$ cannot be a multiple of $1$, hence $1$, $x$, $y$, $xy$ must determine the entire entire point in $G_2 / \SO_{\HH}$.
The Grassmanian of two dimensional subspaces of  $0 \oplus \HH$ is stable under $\GL_2^{\short}$ and so it would be relatively straightforward to study this quotient. 

Note, it is immediately clear that $V$ cannot be totally isotropic because the trace $0$ part of a quaternion algebra does not have a totally isotropic two dimensional subspace.
More generally $V$ cannot be degenerate, if $x$ is isotropic and perpendicular to $y$ which is non-isotropic then $xy$ would also be isotropic but then $y\oplus xy$ would be a two dimensional isotropic subspace which is also perpendicular to $x$, but the perp of a two dimensional subspace of a quaternion algebra is two dimensional.

\section{The family of representations}
\label{s:representation}

In this section we define a family of representations
\[
  \rho=\rho((w,x,y),(a,b,c)): \Gamma \rightarrow G_2
\]
parameterized by points $((w,x,y),(a,b,c)) \in \widetilde{X}_0$, which would be a versal family over an open subset of the character variety $Z$ of all representations $\Gamma \to G_2$. However, the natural description arising from the work above is quite messy to write down. Therefore, we will perform a slight adjustment in order to make our representation easier to describe. As a bonus, this process will also provide some new representations living in other Schubert cells that we omitted from our analysis above.

Towards this end, recall that $G_2(\CC)$ is represented inside $\GL_7(\CC)$ as the group of automorphisms of the octionions $\{i, j, ij, k, ik, jk, (ij)k\}$. We saw above that the moduli space of representations of $\Gamma$ into $G_2$ is covered by a variety depending on parameters $w,x,y,z, a,b,c$ satisfying the conditions $abc-a-b-c=0$ and $w^2+3x^2+y^2+3z^2=1$. However, to simplify our computations, we will introduce new parameters $a_1$, $a_2$, $b_1$, $b_2$, $c_1$, $c_2$,$w$, $x$, $y$, $z$ related to the original variables as follows. Based on the computations in Lemma \ref{lem: gl2reps}, note that the set  of matrices 
  \[
 \begin{pmatrix} 0 & 0 & 0 \\
                           a & 0 & 0 \\
                           0 & b & 0 \\
                           0 & 0 & c \end{pmatrix} 
                           \qquad \text{where}\qquad
  abc-a-b-c=0 
\]
is not itself the orbit of a subgroup of $G_2$. However, it is open inside the orbit of
\[ \left\{ 
\begin{pmatrix} a_1 & 0 &0 & 0 & -a_2 & 0 & 0 \\
                          0 & b_1 &0 & 0 &0 & -b_2 & 0 \\
                          0 & 0 &c_1 & 0 &0 & 0 & -c_2 \\
                          0 & 0 &0 & 1 &0 & 0 & 0 \\
                          a_2 & 0 &0 & 0 &a_1 & 0 & 0 \\
                          0 & b_2 &0 & 0 &0 & b_1 & 0 \\
                          0 & 0 &c_2 & 0 &0 & 0 & c_1 \end{pmatrix} \Bigg| \begin{matrix} a_1^2+a_2^2 = b_1^2+b_2^2 = c_1^2+c_2^2=1, \\[6pt]
                          \begin{pmatrix} a_1 & -a_2 \\ a_2 & a_1 \end{pmatrix} \begin{pmatrix} b_1 & -b_2 \\ b_2 & b_1 \end{pmatrix} \begin{pmatrix} c_1 & -c_2 \\ c_2 & c_1 \end{pmatrix}= \begin{pmatrix} 1 & 0 \\ 0 & 1 \end{pmatrix}\end{matrix} \right\}, \]
which is a maximal torus in $G_2$, acting on the base point $a=b=c=0$.
The points where one of $a_1$, $b_1$, or $c_1$ is zero are not in the coordinate chart in which we have been working, but we shall include them in our new family of representations. This means that we are considering a larger set of representations than we considered in our previous coordinates, but the difference is in codimension (at least) $1$. Thus, let our new parameters satisfy the conditions:
\begin{enumerate}
\item $a_1^2+a_2^2=b_1^2+b_2^2=c_1^2+c_2^2 = 1$,
\item $w^2+3x^2+y^2+3z^2=1$, and
 \item $\stwomat {a_1}{-a_2}{a_2}{a_1}\stwomat {b_1}{-b_2}{b_2}{b_1}\stwomat {c_1}{-c_2}{c_2}{c_1} = I$.
 \end{enumerate}
To go back to the parameters $a,b,c$ we are essentially setting $a = a_2/a_1$, $b = b_2/b_1$ and $c=c_2/c_1$.

In terms of these new parameters, define the representation: 
\begin{align*}
\phi(S) &= \left(\begin{smallmatrix}
a_{1}^{2} - a_{2}^{2} & 0 & 0 & 0 & 2 a_{1} a_{2} & 0 & 0 \\
0 & b_{1}^{2} - b_{2}^{2} & 0 & 0 & 0 & 2 b_{1} b_{2} & 0 \\
0 & 0 & c_{1}^{2} - c_{2}^{2} & 0 & 0 & 0 & 2 c_{1} c_{2} \\
0 & 0 & 0 & -1 & 0 & 0 & 0 \\
2 a_{1} a_{2} & 0 & 0 & 0 & -a_{1}^{2} + a_{2}^{2} & 0 & 0 \\
0 & 2 b_{1} b_{2} & 0 & 0 & 0 & -b_{1}^{2} + b_{2}^{2} & 0 \\
0 & 0 & 2 c_{1} c_{2} & 0 & 0 & 0 & -c_{1}^{2} + c_{2}^{2}
\end{smallmatrix}\right),\\
  \phi(R)  &= \left(\begin{smallmatrix}
1 & 0 & 0 & 0 & 0 & 0 & 0 \\
0 & 1 & 0 & 0 & 0 & 0 & 0 \\
0 & 0 & 1 & 0 & 0 & 0 & 0 \\
0 & 0 & 0 & -\frac{1}{2}& -\frac{3}{2}( w^{2} +3 x^{2} -y^{2} -3 z^{2}) & -3( x y + w z) & 3( w y - 3 x z) \\
0 & 0 & 0 & \frac{1}{2}( w^{2} + 3x^{2} - y^{2} - 3z^{2}) & -\frac{1}{2} & w y - 3 x z & 3( x y +  w z) \\
0 & 0 & 0 & 3(x y + w z) & -3(w y -3 x z) & -\frac{1}{2} & -\frac{3}{2}( w^{2} +3x^{2} -y^{2} -3z^{2}) \\
0 & 0 & 0 & -w y + 3 x z & -3( x y + w z) & \frac{1}{2}( w^{2} + 3 x^{2} - y^{2} - 3z^{2}) & -\frac{1}{2}
\end{smallmatrix}\right).
\end{align*}

\begin{rmk}
  \label{r:phiT}
  If $g$ is a matrix in $G_2$, then its characteristic polynomial takes the form
  \[
X^7+u_1X^6+u_2X^5+(u_1+u_2-u_1^2)X^4-(u_1+u_2-u_1^2)X^3-u_2X^2-u_1X-1.
\]
Notice that $u_1 = -\Tr(g)$. Taking $g = \phi(T)$, where recall $T =\stwomat 1101$, we have
\[
  u_1 = 3(a_2^2+b_2^2+c_2^2)-5.
\]
The expression for $u_2$ is more complicated and depends on all of our parameters. However, since $u_1$ does not depend on $w,x,y,z$, we see that for any specialization of the parameters $a_1$, $a_2$, $b_1$, $b_2$, $c_1$, $c_2$ satisfying our conditions, we have a two-dimensional space of representations such that $\phi(T)$ has a constant characteristic polynomial.

Setting $z=0$ to return to our slice discussed above, this gives a one-parameter family of representations with the characteristic polynomial of $\phi(T)$ fixed, but such that generically in a small neighbourhood where we vary the parameters $w,x,y$, the representations are not isomorphic. This agrees with the fact that the representations on this component of the character variety of representations of the modular group into $G_2$ consists of representations that are not rigid in the sense of Katz \cite{katz}.
\end{rmk}

\begin{prop}
  \label{p:indecomp}
  Suppose that
  \begin{enumerate}
  \item $a_1a_2b_1b_2c_1c_2 \neq 0$;
  \item $a_j \neq \pm b_1,\pm b_2, \pm c_1, \pm c_2$, and $b_j \neq \pm c_1, \pm c_2$ for $j=1,2$;
  \item none of $a_1^2$, $b_1^2$ or $c_1^2$ is equal to $1/2$;
  \item at most one element among $wy-3xz$, $xy+wz$, and $w^2+3x^2-y^2-3z^2$ is zero.
  \end{enumerate}
  Then the representation defined by $\phi$ is indecomposable.
\end{prop}
\begin{proof}
  A representation is reducible if and only if it commutes with some nonscalar transformation.

  Since $\phi(R)$ has $3$ distinct eigenvalues, of multiplicities $3$, $2$ and $2$, its commutator is a linear subspace of matrices of dimension $17$. This subspace is contained inside the larger $25$-dimensional space of matrices of the form
  \[
A = \left(\begin{smallmatrix}
A_{11} & A_{12} & A_{13} & 0 & 0 & 0 & 0 \\
A_{12} & A_{22} & A_{23} & 0 & 0 & 0 & 0 \\
A_{31} & A_{32} & A_{33} & 0 & 0 & 0 & 0 \\
0 & 0 & 0 & A_{44} & A_{45} & A_{46} & A_{47} \\
0 & 0 & 0 & A_{54} & A_{55} & A_{56} & A_{57} \\
0 & 0 & 0 & A_{64} & A_{65} & A_{66} & A_{67} \\
0 & 0 & 0 & A_{74} & A_{75} & A_{76} & A_{77}
\end{smallmatrix}\right)
  \]
  We computed the commutator $[A,\phi(R)]=0$ explicitly, and then restricted to such values of $A$.

  Once restricted to such $A$ values, we then considered the equation $[A,\phi(S)]=0$ on a computer algebra system. After localizing at the equations (1), (2) and (3), it is easy to see that one must have the upper $3\times 3$-block of $A$ being diagonal by examining entries of $[A,\phi(S)]$. After imposing this, then the conditions on $[A,\phi(S)]$, where
  \begin{enumerate}
  \item $A$ satisfies $[A,\phi(R)]=0$ and
  \item the upper $3\times 3$ block of $A$ is diagonal,
  \end{enumerate}
  all factor into expressions of the form:
  \[
   p(a_1,a_2,b_1,b_2,c_1,c_2)\cdot q(w,x,y,z,A_{ij}),
 \]
 for polynomials $p$ and $q$ with rational coefficients.  Moreover, subject to conditions (1), (2) and (3), the factors $p(a_1,a_2,b_1,b_2,c_1,c_2)$ are all invertible. Hence, we computed an ideal $J$ just using the $q(w,x,y,z,A_{ij})$ factors, which is essentially a localization of the equations above. At this point, it was easy to compute a Groebner basis for this simpler ideal $J$ and deduce Proposition 5.1 from this.

 Actually, we performed this computation in several steps: we first used the Groebner basis for $J$ to make $A$ slightly closer to diagonal --- eventually needing to assume condition (4) --- and then we repeated this procedure in several steps. But at the end of this process, conditions (1), (2), (3) and (4) are seen to be sufficient to deduce that only scalar matrices $A$ solve the equations $[A,\phi(R)]=0$ and $[A,\phi(S)]=0$. This concludes the proof.
\end{proof}

Let us summarize some of the important properties of this family of representations in the following theorem:
\begin{thm}
  \label{t:largedegree}
  There exists a nonempty open subset $Z \subseteq X$ in the large component and a cover $\widetilde X_0 \to Z$ described concretely by the family of representations:
\begin{align*}
\phi(S) &= \left(\begin{smallmatrix}
a_{1}^{2} - a_{2}^{2} & 0 & 0 & 0 & 2 a_{1} a_{2} & 0 & 0 \\
0 & b_{1}^{2} - b_{2}^{2} & 0 & 0 & 0 & 2 b_{1} b_{2} & 0 \\
0 & 0 & c_{1}^{2} - c_{2}^{2} & 0 & 0 & 0 & 2 c_{1} c_{2} \\
0 & 0 & 0 & -1 & 0 & 0 & 0 \\
2 a_{1} a_{2} & 0 & 0 & 0 & -a_{1}^{2} + a_{2}^{2} & 0 & 0 \\
0 & 2 b_{1} b_{2} & 0 & 0 & 0 & -b_{1}^{2} + b_{2}^{2} & 0 \\
0 & 0 & 2 c_{1} c_{2} & 0 & 0 & 0 & -c_{1}^{2} + c_{2}^{2}
\end{smallmatrix}\right),\\
  \phi(R)  &= \left(\begin{smallmatrix}
1 & 0 & 0 & 0 & 0 & 0 & 0 \\
0 & 1 & 0 & 0 & 0 & 0 & 0 \\
0 & 0 & 1 & 0 & 0 & 0 & 0 \\
0 & 0 & 0 & -\frac{1}{2}& -\frac{3}{2}( w^{2} +3 x^{2} -y^{2}) & -3 x y  & 3 w y \\
0 & 0 & 0 & \frac{1}{2}( w^{2} + 3x^{2} - y^{2}) & -\frac{1}{2} & w y & 3x y \\
0 & 0 & 0 & 3x y & -3w y & -\frac{1}{2} & -\frac{3}{2}( w^{2} +3x^{2} -y^{2}) \\
0 & 0 & 0 & -w y & -3 x y & \frac{1}{2}( w^{2} + 3 x^{2} - y^{2}) & -\frac{1}{2}
\end{smallmatrix}\right).
\end{align*}
where
\begin{enumerate}
\item $a_1^2+a_2^2=b_1^2+b_2^2=c_1^2+c_2^2 = 1$,
\item $w^2+3x^2+y^2=1$, and
 \item $\stwomat {a_1}{-a_2}{a_2}{a_1}\stwomat {b_1}{-b_2}{b_2}{b_1}\stwomat {c_1}{-c_2}{c_2}{c_1} = I$.
 \end{enumerate}
This map $\widetilde X_0 \to Z$ is generically finite of degree $384$.
\end{thm}
\begin{proof}
  This is a consequence of Theorem \ref{thm:genericfiniteall} and the discussion above about introducing the parameters $a_1$, $a_2$ et cetera in place of $a$, $b$, $c$ as in Theorem \ref{thm:genericfiniteall}.
\end{proof}

While the cover above arises from a natural construction, its degree is quite large. We shall now introduce a factorization via a reparameterization of the $w$, $x$ and $y$ variables that will eliminate some of the excessive degrees of symmetry above. Let us now write:
\begin{align*}
  w_1 &= wy,&  x_1 &= xy,&  y_1 &= y^2,
\end{align*}
so that
\[
\phi(R)  = \left(\begin{smallmatrix}
1 & 0 & 0 & 0 & 0 & 0 & 0 \\
0 & 1 & 0 & 0 & 0 & 0 & 0 \\
0 & 0 & 1 & 0 & 0 & 0 & 0 \\
0 & 0 & 0 & -\frac{1}{2}& -\frac{3}{2}( 1 -2y_1) & -3 x_1  & 3 w_1 \\
0 & 0 & 0 & \frac{1}{2}(1 - 2y_1) & -\frac{1}{2} & w_1 & 3x_1\\
0 & 0 & 0 & 3x_1 & -3w_1 & -\frac{1}{2} & -\frac{3}{2}( 1 -2y_1) \\
0 & 0 & 0 & -w_1 & -3 x_1 & \frac{1}{2}( 1 - 2y_1) & -\frac{1}{2}
\end{smallmatrix}\right).
\]

\subsection{Proof of Theorem \ref{t:main1}}
  \label{proofmain1}
  Notice that after our reparameterization, condition (2) of Theorem \ref{t:largedegree} becomes the condition (2) of Theorem \ref{t:main1}. It remains to check the degree is as stated, and that the Galois group is what we say it is in Theorem \ref{t:largedegree}.

  The reparameterization is of degree $8$, and so it reduces the total degree of the map in Theorem \ref{t:largedegree} by a factor of $8$. This results in a degree of $48$ as claimed.

  As for the symmetries of the cover, observe that if we write $t(\phi)$ and $u(\phi)$ for the values of the invariants $t$ and $u$ from Section \ref{s:computational} evaluated on our representation, then
  \begin{align*}
  t(\phi) &=-3\frac{(a_2b_1c_1)^2+(a_1b_2c_1)^2+(a_1b_1c_2)^2}{(a_1b_1c_1)^2},\\
  u(\phi) &= -3\frac{a_2b_2c_1+a_2b_1c_2+a_1b_2c_2}{a_1b_1c_1}.
  \end{align*}
  These are visibly invariant under permutation of the pairs $(a_1,a_2)$, $(b_1,b_2)$ and $(c_1,c_2)$. They are also preserved by the following sign changes: one can swap $(a_1,a_2)$ and $(b_1,b_2)$ with their negatives, and similarly for the other possible pairs. Then there is the symmetry given by changing the sign of all of $a_2$, $b_2$ and $c_2$ simultaneously. This gives another $8$ symmetries, for a total group of symmetries of degree $48$, which agrees with the degree of our map. Notice also that these symmetries respect the defining equations (1), (2) and (3) of Theorem \ref{t:main1}. Therefore, this exhausts all of the symmetries of this cover.

  By construction of our invariants $t$ and $u$, if $\phi_1$ and $\phi_2$ are two equivalent specializations, then we have $t(\phi_1) = t(\phi_2)$ and $u(\phi_1)= u(\phi_2)$ as in the last statement of Theorem \ref{t:main1}. The further claim about the equality of the $w_1$, $x_1$ and $y_1$ coordinates of the two specializations almost follows from Theorem \ref{thm:genericfiniteall}, except that in that theorem the map to the invariants had degree $2$ ambiguity coming from the map $Z \to Y$ of that theorem. The issue at this point is that we have not every given a concrete description of $Z$ per se, but rather we have described covers and quotients in explicit fashion. We now eliminate the ambiguity and give a concrete description of $Z$, which is an open subset of our moduli space.

  We claim that
  \begin{equation}
\label{eq:Z}
    Z = \{(t,u,w_2,x_2,y_1) \mid w_2+x_2+y_1^2=y_1\}.
\end{equation}
To see this, observe that if we set $w_2 = w_1^2$ and $x_2 = 3x_1^2$, then we have maps:
  \[
\widetilde X_0 \to \{(t,u,w_2,x_2,y_1) \mid w_2+x_2+y_1^2=y_1\} \to Z \to Y,
\]
which factors $f\colon \widetilde X_0 \to Z$, where $f$ is as defined in Theorem \ref{thm:genericfiniteall}. Note that in the notation of Theorem \ref{thm:genericfiniteall} we have $w_2 = w_0y_0$ and $x_2 =x_0y_0$.

The last map above $Z \to Y$ is degree $2$, but an examination of the proof of Theorem \ref{thm:genericfiniteall} shows that so is the map 
\[
\{(t,u,w_2,x_2,y_1) \mid w_2+x_2+y_1^2=y_1\} \to Y.
\]
In particular, this shows that $Z$ is as stated in equation \eqref{eq:Z}.

Since $Z$ is an open subset of the moduli space of equivalence classes of representations $\Gamma \to G_2$, this now proves the last statement of Theorem \ref{t:main1}.  

\begin{rmk}
  \label{rmk:degree2discrepancy}
  The modular group has an outer automorphism $\iota$ given by the map
  \[
 \iota\stwomat abcd= \stwomat{a}{-b}{-c}d.
\]
Likewise, the family of representations $\phi$ of Theorem \ref{t:main1} admits an outer automorphism mapping $(w_1,x_1,y_1) \mapsto (-w_1,-x_1,1-y_1)$. A direct computation shows that this latter symmetry takes $\phi(R)$ to $\phi(R^2)$, and precomposition with $\iota$ does the same, thus these symmetries are equal. They are likewise equal to the covering automorphism of the map $Z\to Y$ discussed in Theorem \ref{thm:genericfiniteall}.
\end{rmk}

\section{Factoring through a maximal subgroup}
\label{s:maximal}
The maximal subgroups of $G_2(\bF_q)$ have been classified in \cite{cooperstein,kleidman} up to conjugacy. For $q=p^\alpha$ and $p\neq 3$ there is a unique conjugacy class of the following types of maximal subgroups:
\begin{enumerate}
\item An extension of an $A_2$ subgroup by its outer automorphism.
\item A $D_2$ subgroup. This is the group $\SO_\HH$.
\item The parabolic subgroup for the long root $A_1$ subgroup.
\item The parabolic subgroup for the short root $A_1$ subgroup.
\item $G_2(p^\beta)$ for $\alpha=\ell\beta$ with $\ell$ a prime. These do not exist when $q=p$.
\item A group isomorphic to $\PGL_2(q)$. 
\item A group $2^3\cdot L_3(2) = 2^3\cdot L_2(7)$ which has order $1344=8\cdot 168 = 2^{6}\cdot 3 \cdot 7$.
\item A group $L_2(8) =  {^2}G_2(3)'$ which has order $504 = 2^3\cdot3^2\cdot7$.
(this only occurs if $\bF[w]\subseteq\bF_q$ with $w^3-3w+1=0$ and is only maximal with equality.)

\item A group $L_2(13)$ which has order $1092 = 2^2\cdot3\cdot7\cdot13$.
(this only occurs if $\bF[w]\subseteq\bF_q$ with $w^2-13=0$ and is only maximal with equality.)
\item A group $G_2(2)=U_3(3):2$ which has order $12096 = 2^6\cdot 3^3 \cdot 7$ 
(this is only maximal if $q$ is prime.)
\item A group $J_1$ which only occurs for $p=11$. This group has order $175560= 2^3\cdot3\cdot5\cdot7\cdot11\cdot19$.
\end{enumerate}

In Section \ref{s:obstructions} below we discuss the general problem of classifying representations that factor through maximal subgroups of $G(\FF_q)$. In the following sections we then specialize to the groups above to determine sufficient conditions for when our family $\phi$ discussed in Section \ref{s:representation} is surjective onto $G_2(\FF_p)$. We restrict to $q=p$ being prime mainly to avoid having to discuss case (5) above.

\subsection{Obstructions to surjectivity}
\label{s:obstructions}
The primary obstruction to surjectivity of a map $\varphi : \PSL_2(\ZZ) \rightarrow G(\FF_q)$ is that the image is contained in one of the maximal subgroups of $G(\FF_q)$.

For each maximal subgroup $M$ of $G(\FF_q)$ , and each component of the moduli space of maps $\PSL_2(\ZZ)$, that is each pair of a conjugacy class of an order $2$ and order $3$ element, determines
a subvariety of moduli space of maps $ \PSL_2(\ZZ)$ to $G$ over an algebraically closed field. Indeed, fix $\tilde{\alpha_2} =w_2^{-1} \alpha_2 w_2 \in M$ a conjugate of $\alpha_2$ and $\tilde{\alpha_3} = w_3\alpha_3w_3^{-1}\in M$ then we map
\[  (M/ Z_M(\tilde{\alpha_2}) \times M /Z_M(\tilde{\alpha_3}))/M \rightarrow  (G/Z_G(\alpha_2) \times G /Z_G(\alpha_3) )/G\]
by  
\[  ( m_2, m_3)  \mapsto   (m_2w_2,m_3w_3). \]
In particular, the image of $Z_M(\tilde{\alpha_3}) \backslash M /Z_M(\tilde{\alpha_2})$ in $Z_G(\alpha_3) \backslash G /Z_G(\alpha_2)$ is simply
\[Z_G(\alpha_3)   \backslash w_3 M w_2/Z_G(\alpha_2). \]
We note that there are only finitely many conjugacy classes of maximal subgroups, these are listed below, and in each only finitely many conjugacy classes of order $2$ and order $3$ elements.

In the cases where the inclusion of $M$ arises as the inclusion of an algebraic group then we may take $M$ to be one of the standard forms described in Section \ref{ss:relsub} so that each conjugacy class of order $2$ and $3$ element has a representative from a maximal torus. In this way the elements $w_2$ and $w_3$ can be assumed to be Weyl group elements from $G$.
The element $w_2$ can always be one of the three cycles from the Weyl group, as these act transitively on order $2$ elements.
The element $w_3$ can always be one of the six cycles from the Weyl group, as these act transitively on both conjugacy classes of order $3$ elements from $A$.

\begin{rmk}
Because many of the maximal subgroups have a unique conjugacy class of order $2$ and $3$ element they cannot have image in both $X$ and $Y$. 
Whether they have image in $X$ or $Y$ will depend on which order $3$ element from $G$ the order $3$ element from $M$ is conjugate to. 
This can be distinguished using the character of the relevant $7$-dimensional representation.
\end{rmk}

\subsection{Determining conditions for reducibility}
\label{ss:reducible} Our aim now is to determine some Zariski open conditions that ensure our represention $\phi$ surjects onto $G_2(\FF_p)$. For this, it is sufficient to know that the image is not contained in a maximal subgroup. For three classes of maximal subgroups, namely the $A_2$, $\SO_\HH$ and parabolic subgroups, we can determine such conditions is a somewhat uniform manner that we describe here.

The uniform feature of these three families of maximal subgroups is that their natural $7$-dimensional representation coming from their embedding in $G_2$ has a nontrivial subrepresentation of dimension at most $3$.

Let us now describe how one can determine such Zariski open conditions to rule out having a nontrivial subrepresentation of small dimension. To rule out one-dimensional subrepresentations, one is looking at common eigenvectors, which is a relatively straightforward case. Therefore, let us describe in detail how to handle the case of subrepresentations of dimension $2$, and the reader can extend this naturally to the case of subrepresentations of dimension $3$ or higher.

We will work on cells in the Grassmanian of $2$-planes $\Gr(2,7)$. The standard chart is
\[
  C=\left(\begin{matrix}
      1 & 0 & A_1&A_2&A_3&A_4&A_5\\
      0 & 1 & A_6&A_7&A_8&A_9&A_{10}
\end{matrix}\right)
\]
Our condition is that  $\phi(S)C^T$ and $\phi(R)C^T$ are still contained in the chart defined by $C$. This leads to linear conditions on the variables $A_j$ where the coefficients are polynomials in our ambient variables $a_1$, $a_2$, et cetera. These equations define an ideal $I$ in the ring of variables involving the $a_1$, $a_2$ et cetera, along with our new variables $A_j$ introduced in $C$. Ideally we would like to find a primary decomposition for $I$ but the naive approach may not always be computationally feasible, and so we employ a somewhat ad hoc trick.

The idea now is that we already know of some algebraic conditions on our original variables $a_1$, $a_2$, et cetera, that lead to reducible representations. For example, see the conditions in Proposition \ref{p:indecomp}, specifically conditions (1) and (2). We proceed by computing a Groebner basis for $I$, and we observe that many of the basis elements are divisible by equations in the ideal defined by conditions (1) and (2) of Proposition \ref{p:indecomp}. Therefore, we can perform a manual sort of localisation away from these conditions by simply dividing these basis vectors by the corresponding conditions. This then leads to a new ideal $I'$. If we are not yet able to primary decompose $I'$, we then repeat this procedure. Somewhat surprisingly, to us, this ad hoc computation allowed us to formulate the results in the following three subsections.

This of course only treats the standard cell, but one repeats the procedure a finite number of times for the other cells.

\subsection{Factoring through an  $A_2$ subgroup}
\label{ss:factoringA2}
The $A_2$ subgroup is the group $\SU_M$, and in this case things are particularly simple because $\phi(R)$ and $\phi(S)$ have a common eigenvector with eigenvalue $1$. Keeping in mind the relation $3x^2+y^2+3z^2+w^2=1$, one can explicitly check that such an eigenvector of $\phi(R)$ is of the form $(\alpha, \beta, \gamma, 0,0,0,0)^t$. Multiplying  by $\phi(S)$  then  implies that at least one of the $a_1a_2, b_1b_2,$ and $c_1c_2$ equal zero, which gives that a necessary (and sufficient)  condition  for $\phi$ to factor through $\SU_M$ is that
$$a_1 a_2 b_1 b_2 c_1 c_2 = 0.$$ 
This is the same equation (1) that arose in Proposition \ref{p:indecomp}.

\subsection{Factoring through $\SO_\HH$}
\label{ss:factoringSOH}
An $\SO_{\HH}$ subgroup has a $3$-dimensional subrepresentation. Localizing away from $a_1 a_2 b_1 b_2 c_1 c_2 = 0$, which as we have seen factors through $\SU_M$, one can proceed determining Zariski open conditions for ruling such factorization out by using the method outlined in Section \ref{ss:reducible}.

Alternatively, a nontrivial $\SO_{\HH}$ subgroup of $G_2$ has a nontrivial centralizer. Conversely, if a subgroup of $G_2$ has a nontrivial centralizer and is not contained in an $A_2$-subgroup, then it is contained in an $\SO_{\HH}$ subgroup (this follows by the classification of centralizers of elements in $G_2$). The following nine components give rise to subgroups with nontrivial centralizer and therefore the map factors through $\SO_\HH$:
\begin{enumerate}
\item $(x-z, w+y, 2y^2+6z^2-1)$,
\item $(x+z, w-y, 2y^2+6x^2-1)$,
\item $(\sqrt{3}x+y, \sqrt{3}z+w, 2w^2+2y^2-1)$,
\item $(\sqrt{3}x-y, \sqrt{3}z-w, 2w^2+2y^2-1)$,
\item $(w,x, y^2+3z^2-1)$,
\item $(y,z,w^2+3x^2-1)$,
\item $(a_1-b_1, a_2-b_2, 3xz-yw)$ or $(a_1+b_1, a_2+b_2, 3xz-yw)$,
\item $(b_1- c_1, b_2- c_2,2w^2+2w^2-1,2y^2+6z^2-1)$ or\\ $(b_1+ c_1, b_2+ c_2,2w^2+2w^2-1,2y^2+6z^2-1)$,
\item $(a_1- c_1, a_2- c_2, xy+zw)$ or $(a_1+ c_1, a_2+ c_2, xy+zw)$.
\end{enumerate}
Notice that the equations above explain in more detail how parts (2), (3) and (4) in Proposition \ref{p:indecomp} arise.

\subsection{Factoring through a parabolic}

The parabolic subgroups of $G_2$ are associated to isotropic subspaces of $\OO_7$. 

By running the calculations as described in section \ref{ss:reducible}, and localizaing at $a_1$, $a_2$, $b_1$, $b_2$, $c_1$, and $c_2$ one can easily show that if there is a $1$ or $3$ dimensional isotropic subspace then $a_1a_2b_1b_2c_1c_2=0$, and hence the map already factors through an $A_2$ subgroup above.

Similarly, if one localizes at   $a_1,a_2,b_1,b_2,c_1,c_2$ and $a_1-b_1$ ,$a_2-b_2$, $a_1+b_1$, $a_2+b_2$, $a_1-c_1$, $a_2-c_2$, $a_1+c_1$, $a_2+c_2$, $b_1-c_1$, $b_2-c_2$,$b_1+c_1$, and $b_2+c_2$, one finds that any representation with a two dimensional isotropic subspace factors through a group satisfying one of these conditions.

To understand why we include for example $a_2+b_2$ notice that $a_2+b_2=0$ implies $a_1=\pm b_1$, and any point with $a_1=b_1$, and $a_2=-b_2$ would imply $c_2=0$, and that points with $a_1=-b_1$, and $a_2=-b_2$ are equivalent to points $a_1=b_1$, and $a_2=b_2$ since $\phi$ is invariant under negating each of $(a_1,a_2)$, $(b_1,b_2)$, or $(c_1,c_2)$.

This would already give sufficient conditions to not factor through a parabolic, however, we choose to refine these by now imposing the respective conditions, for example $b_1=c_1$ and $b_2=c_2$. Note that $b_1=\pm c_1$ implies $b_2=\pm c_2$. We continue to localize away from $a_1-b_1$ ,$a_2-b_2$, $a_1+b_1$, $a_2+b_2$, $a_1-c_1$, $a_2-c_2$, $a_1+c_1$, $a_2+c_2$ and add the equation $2a_1+1$, which the equations would now imply is non-vanishing.

Repeating this for the various options, and also considering the case $a_1=b_1=c_1$ and $a_2=b_2=c_2$, we find that the following ideals give subgroups factoring through parabolics:
\begin{enumerate}
\item $(a_1 - b_1, a_2 - b_2, 2wy - 6xz + 1)$ or $(a_1 + b_1, a_2 + b_2, 2wy - 6xz + 1)$,
\item $(a_1 - b_1, a_2 - b_2, 2wy - 6xz - 1)$ or $(a_1 + b_1, a_2 + b_2, 2wy - 6xz - 1)$,
\item $(a_1 - c_1, a_2 - c_2, 6xy + 6wz + \sqrt{3})$ or $(a_1 + c_1, a_2 + c_2, 6xy + 6wz + \sqrt{3})$,
\item $(a_1 - c_1, a_2 - c_2, 6xy + 6wz - \sqrt{3})$ or $(a_1 + c_1, a_2 + c_2, 6xy + 6wz - \sqrt{3})$,
\item $(b_1 - c_1, b_2 - c_2, y^2 + 3z^2 - 1, w^2 + 3x^2)$ or $(b_1 + c_1, b_2 + c_2, y^2 + 3z^2 - 1, w^2 + 3x^2)$,
\item $(b_1 - c_1, b_2 - c_2, y^2 + 3z^2, w^2 + 3x^2-1)$ or $(b_1 + c_1, b_2 + c_2, y^2 + 3z^2, w^2 + 3x^2-1)$,
\item $(a_1-b_1, a_1-c_1, a_2-b_2, a_2 - c_2)$, $(a_1+b_1, a_1-c_1, a_2+b_2, a_2 - c_2)$, \\
  $(a_1-b_1, a_1+c_1, a_2-b_2, a_2 + c_2)$, or $(a_1+b_1, a_1+c_1, a_2+b_2, a_2 + c_2)$.
\end{enumerate}

\subsection{Factoring through $\PGL_2$}
\label{ss:factoringPGL2}
One of the most difficult cases to analyze below is the case of representations that factor through a copy of $\PGL_2$. Such representations involve matrices whose characteristic polynomial factors in the following form:
\[
P(X) = (X-1)(X-\lambda)(X-\lambda^2)(X-\lambda^3)(X-\lambda^{-1})(X-\lambda^{-2})(X-\lambda^{-3}).
\]
If we write this polynomial as
\[
  P(X) = X^7+g_1X^6+g_2X^5+g_3X^4-g_3X^3-g_2X^2-g_1X-1,
\]
then these coefficients satisfy the following equations:
\begin{enumerate}
\item $g_1^2-g_1-g_2+g_3 = 0$,
\item $g_1^2g_2-g_1g_2^2-g_2^3+2g_1^2g_3-g_1g_3^2 = 0$,
\item $g_2^4-g_1g_2^2g_3-g_2^3g_3+g_1g_2g_3^2-g_1g_3^3-g_2^3+g_2^2g_3+2g_2g_3^2-g_3^3 = 0$.
\end{enumerate}
Note that equation (1) is the same as for those polynomials arising from $G_2$, as expected. If we use that equation to eliminate $g_3$, then $g_1$ and $g_2$ satisfy the higher degree equation:
\[
g_1^5-2g_1^3g_2-g_1^3-g_1^2g_2+2g_1g_2^2+g_2^3 = 0.
\]
It follows that as long as the coefficients of the characteristic polynomial of $\phi(T)$ --- or any element $\phi(\gamma)$ --- avoid this hypersurface, then our representation does not factor through a $\PGL_2$. This is a computation that is easy to check in examples on a computer, say, but which seems difficult to express more simply in an explicit fashion in terms of the variables $a_1$, $a_2$, $b_1$, $b_2$, $c_1$, $c_2$, $w$, $x$, $y$, $z$.

\subsection{Factoring through a subgroup of order independent of $p$}
\label{ss:finite}
From the classification of maximal subgroups of $G_2(\FF_p)$ whose order is independent of $p$, for $p\neq 11$, every element in such a subgroup must have order equal to one of the following:
\[
1,~2,~3,~4,~6,~7,~8,~9,~12,~13
\]
When $p=11$, it is also possible to find elements of order $5$, $15$ and $19$, which arise from the Janko group $J_1$. One can classify the characteristic polynomials of elements of these orders inside $G_2$ --- see Appendix \ref{appendix1} for this data. To ensure that $\phi$ does not have image contained inside one of these maximal subgroup, we simply impose the conditions that $\phi(T)$ does not have characteristic polynomial equal to any of these possibilities.

Recall from Remark \ref{r:phiT} that if $u_1 = -\Tr(\phi(T))$ then
\[
  u_1=3(a_2^2+b_2^2+c_2^2)-5.
\]
Therefore, in the cases where $p\neq 11$, we can easily rule out the cases of factoring through a finite subgroup of $G_2$ by imposing the conditions $u_1\neq \tau$ for each value of $\tau$ arising in the $X^6$ term of the characteristic polynomials in Appendix \ref{appendix1}. When $p=11$ we use slightly more data from the characteristic polynomial, to increase the number of examples that one could consider.

\subsection{Proof of Theorem \ref{t:main2}}
Recall that $\widetilde X_0$ is described in Theorem \ref{thm:genericfiniteall} above. We have discussed the classification of maximal subgroups of $G_2(\FF_p)$ up to conjugacy. Condition (1) of Theorem \ref{t:main2} amounts to ensuring that the image of $\phi$ does not land inside of an $A_2$ subgroup. Condition (2) ensures that $\phi$ does not have image inside of a parabolic subgroup. Then conditions (2) through (4) together imply that the image is not contained in an $\SO_{\HH}$ subgroup. Condition (5) rules out the case of $\PGL_2$ subgroups. Finally, condition (6) rules out most of the remaining subgroups whose orders are independent of $p$, using the data from Appendix \ref{appendix1}, and condition (7) treats the case of the Janko group $J_1$ when $p=11$. Since we have restricted to considering representations taking values in $G_2(\FF_p)$ for primes $p$, we can ignore the remaining maximal subgroups showing up in \cite{kleidman} that only arise when considering larger finite fields of prime power order, such as $G_2(\FF_p)\subseteq G_2(\FF_q)$. Therefore, the algebraic conditions in Theorem \ref{t:main2} are enough to ensure that $\phi$ surjects on $G_2(\FF_p)$.

\appendix
\section{The maximal subgroups of $G_2(\FF_q)$ of bounded order}
\label{appendix1}

Most of the maximal subgroups discussed in Section \ref{s:maximal} above fit into families that depend on $q$, but there are a number of examples from Section \ref{ss:finite} whose orders are independent of $q$. In this appendix we collect data about the possible characteristic polynomials of elements in these groups for use in the formulation and proof of Theorem \ref{t:main2}. More precisely, the characteristic polynomials appearing in the tables below are used to formulate the conditions in parts (6) and (7) of Theorem \ref{t:main2}. 

In this Appendix we now set $q=p^\alpha$ and $p\neq 3$ and to ease notation write $G_2 = G_2(\FF_q)$ below.

\subsection{$2^3 \cdot L_3(2)$} The group $G_2$ contains maximal subgroups isomorphic with $2^3\cdot L_3(2) = 2^3\cdot L_2(7)$, a group of order $2^{6}\cdot 3 \cdot 7$. A representative for this conjugacy class of subgroups is described as the subgroup  which permutes the lines spanned by $i$, $j$, $ij$, $k$, $ik$, $jk$, and  $(ij)k$ in the octonions. This subgroup contains two conjugacy classes of elements of order $2$ and a unique conjugacy class of elements of order $3$.
 
When the image of $\phi$ is contained in an $2^3\cdot L_3(2)$ subgroup of $G_2$, the possible orders for $\phi(T)$ are $2$, $3$, $4$, $6$, $7$, and $8$ which is all of the non-trivial orders of elements in the group. The corresponding characteristic polynomials are contained in Table \ref{L32}.

\begin{center}
  \begin{table}[h!]
    \label{L32}
    \begin{tabular}{l|l}
      Characteristic polynomial of $\phi(T)$ & Order of $\phi(T)$\\\hline
      $ x^7 + x^6 - 3x^5 - 3x^4 + 3x^3 + 3x^2 - x - 1$& $2$\\
      $x^7 - x^6 - 2x^4 + 2x^3 + x - 1$& $3$\\
      $ x^7 - 3x^6 + 5x^5 - 7x^4 + 7x^3 - 5x^2 + 3x - 1$ & $4$\\
      $x^7 + x^6 + x^5 + x^4 - x^3 - x^2 - x - 1$ & $4$\\
      $x^7 + x^6 - x - 1$ & $6$\\
      $x^7-1$& $7$\\
      $x^7 - x^6 + x^5 - x^4 + x^3 - x^2 + x - 1$& $8$\\
      $x^7 + x^6 - x^5 - x^4 + x^3 + x^2 - x - 1$& $8$
    \end{tabular}
\caption{Characteristic polynomials for $2^3\cdot L_3(2)$.}
  \end{table}
\end{center}

\subsection{$L_2(8)$}  The group $G_2$ also contains subgroups isomorphic with $L_2(8) =  {^2}G_2(3)'$, a group of order $2^3\cdot3^2\cdot7$.  These groups only arise if $\bF[\alpha]\subseteq\bF_q$ with $\alpha^3-3\alpha+1=0$ and, these subgroups are only maximal if $\bF[\alpha] = \bF_q$. There is a unique conjugacy class of elements of order $2$, and a unique conjugacy class of elements of order $3$.


To express this in characteristic zero will require something acting as a $9$th root of unity in our torus $A$, introduced in Section \ref{ssec:mt}. Note that our torus contains $3$rd roots of unity and $\alpha$, with $\alpha^3-3\alpha+1=0$, allows us to express the $9$th roots of unity in the torus $A$. The possible orders for $\phi(T)$ are $2$, $7$, and $9$. The non-trivial orders of elements in the group are $2$, $3$, $7$ and $9$. Options for characteristic polynomials are found in Table \ref{L28}. The line in Table \ref{L28} with $\alpha$ appearing really corresponds to three examples, where $w$ can be replaced by its Galois conjugates $\alpha^2-\alpha-2$ and $-\alpha^2+2$.

\begin{center}
  \begin{table}[h!]
    \label{L28}
    \begin{tabular}{l|l}
      Characteristic polynomial of $\phi(T)$ & Order of $\phi(T)$\\\hline
     $ x^7 + x^6 - 3x^5 - 3x^4 + 3x^3 + 3x^2 - x - 1$ & $2$\\
     $x^7 - x^6 - 2x^4 + 2x^3 + x - 1$ & $3$\\
     $x^7-1$ & $7$\\
       $x^7 + \alpha x^6 + (\alpha^2-1)x^5 + (\alpha - 1)x^4 + (1-\alpha)x^3 + (1-\alpha^2)x^2 - \alpha x - 1$& $9$
    \end{tabular}
\caption{Characteristic polynomials for $L_2(8)$.}
  \end{table}
\end{center}

\subsection{$L_2(13)$} Inside $G_2$ one also finds maximal subgroups isomorphic with $L_2(13)$, which has order $2^2\cdot3\cdot7\cdot13$. These examples only arise when $\bF_p[\sqrt{13}]\subseteq\bF_q$, and these subgroups are maximal only if $\bF_q = \bF_p[\sqrt{13}]$. This subgroups contains a unique conjugacy class of elements of order $2$, and unique conjugacy class of elements of order $3$.

The possible orders for $\phi(T)$ inside this subgroup are $2$, $3$, $6$, $7$, and $13$, which includes all of the possible non-trivial orders of elements in the subgroup.  Options for characteristic polynomials are found in Table \ref{L213}.

\begin{center}
  \begin{table}[h!]
    \label{L213}
    \begin{tabular}{l|l}
      Characteristic polynomial of $\phi(T)$ & Order of $\phi(T)$\\\hline
      $ x^7 + x^6 - 3x^5 - 3x^4 + 3x^3 + 3x^2 - x - 1$& $2$\\
      $x^7 - x^6 - 2x^4 + 2x^3 + x - 1$& $3$\\
      $x^7 + x^6 - x - 1$& $6$\\
      $x^7-1$ & $7$\\
      $ x^7 + (1/2)(\pm \sqrt{13}-1)x^6 + (1/2)(3\mp \sqrt{13})x^5 + (1/2)(\pm\sqrt{13}-5)x^4$ & $13$\\
      $- (1/2)(\pm \sqrt{13}-5)x^3 -  (1/2)(3\mp \sqrt{13})x^2 -  (1/2)(\pm \sqrt{13}-1)x - 1$ & 
    \end{tabular}
\caption{Characteristic polynomials for $L_2(13)$.}
  \end{table}
\end{center}

\subsection{$U_3(3):2$}
There also exist maximal subgroups of $G_2$ isomorphic with $U_3(3):2 = G_2(2)$, which has order $2^6\cdot 3^3 \cdot 7$. These examples are only maximal if $q=p$ is prime. They contain a unique conjugacy class of elements of order $2$, and two conjugacy classes of elements of order $3$.

Inside these subgroups, the possible orders for $\phi(T)$ are $2$, $3$, $4$, $6$, $7$, $8$, and $12$, which includes all of the possible non-trivial orders of elements in the group. Options for the characteristic polynomials are contained in Table \ref{U32}
\begin{center}
  \begin{table}[h!]
    \label{U32}
    \begin{tabular}{l|l}
      Characteristic polynomial of $\phi(T)$ & Order of $\phi(T)$\\\hline
      $ x^7 + x^6 - 3x^5 - 3x^4 + 3x^3 + 3x^2 - x - 1$& $2$\\
      $ x^7 + 2x^6 + 3x^5 + x^4 - x^3 - 3x^2 - 2x - 1$& $3$\\
      $ x^7 - x^6 - 2x^4 + 2x^3 + x - 1$& $3$\\
      $ x^7 - 3x^6 + 5x^5 - 7x^4 + 7x^3 - 5x^2 + 3x - 1$& $4$\\
      $x^7 + x^6 + x^5 + x^4 - x^3 - x^2 - x - 1$& $4$\\
      $x^7 + x^6 - x - 1$& $6$\\
      $x^7 - 2x^6 + 3x^5 - 3x^4 + 3x^3 - 3x^2 + 2x - 1$& $6$\\
      $x^7 - 1$& $7$\\
      $x^7 - x^6 + x^5 - x^4 + x^3 - x^2 + x - 1$& $8$\\
      $x^7 + x^6 - x^5 - x^4 + x^3 + x^2 - x - 1$& $8$\\
      $ x^7 - x^5 - x^4 + x^3 + x^2 - 1$& $12$
    \end{tabular}
\caption{Characteristic polynomials for $U_3(2):2$.}
  \end{table}
\end{center}

\subsection{$J_1$} When $p=11$, the group $G_2$ contains maximal subgroups isomorphic with the Janko group $J_1$. This group has order $2^3\cdot3\cdot5\cdot7\cdot11\cdot19$ and it contains unique conjugacy classes of elements of order $2$ and $3$. A priori, just by considering the possible choices for conjugacy class representatives of these order $2$ and $3$ classes, the possible orders for $\phi(T)$ are $2$, $3$, $5$, $6$, $7$, $10$, $11$, $15$, $19$. The corresponding characteristic polynomials are found in Table \ref{J1}.

\begin{center}
  \begin{table}[h!]
    \label{J1}
    \begin{tabular}{l|l}
    Characteristic polynomial of $\phi(T)$ & Order of $\phi(T)$\\\hline
$x^{7} + x^{6} + 8 x^{5} + 8 x^{4} + 3 x^{3} + 3 x^{2} + 10 x + 10$ & $2$ \\
$x^{7} + 10 x^{6} + 9 x^{4} + 2 x^{3} + x + 10$ & $3$ \\
$x^{7} + 8 x^{6} + x^{5} + 10 x^{2} + 3 x + 10$ & $5$ \\
$x^{7} + 4 x^{6} + x^{5} + 10 x^{2} + 7 x + 10$ & $5$ \\
$x^{7} + x^{6} + 10 x + 10$ & $6$ \\
$x^{7} + 10$ & $7$ \\
$x^{7} + 6 x^{6} + 9 x^{5} + x^{4} + 10 x^{3} + 2 x^{2} + 5 x + 10$ & $10$ \\
$x^{7} + 2 x^{6} + 6 x^{5} + 4 x^{4} + 7 x^{3} + 5 x^{2} + 9 x + 10$ & $10$ \\
$x^{7} + 4 x^{6} + 10 x^{5} + 9 x^{4} + 2 x^{3} + x^{2} + 7 x + 10$ & $11$ \\
$x^{7} + 6 x^{6} + 3 x^{5} + 6 x^{4} + 5 x^{3} + 8 x^{2} + 5 x + 10$ & $15$ \\
$x^{7} + 3 x^{6} + 2 x^{5} + 7 x^{4} + 4 x^{3} + 9 x^{2} + 8 x + 10$ & $15$ \\
$x^{7} + 9 x^{6} + 5 x^{4} + 6 x^{3} + 2 x + 10$ & $19$ \\
$x^{7} + 7 x^{6} + 10 x^{5} + x^{4} + 10 x^{3} + x^{2} + 4 x + 10$ & $19$ \\
$x^{7} + 4 x^{6} + 7 x^{5} + 6 x^{4} + 5 x^{3} + 4 x^{2} + 7 x + 10$ & $19$ 
    \end{tabular}
    \caption{Characteristic polynomials for $J_1$.}
\end{table}
\end{center}

\bibliography{refs}
\bibliographystyle{plain}
\end{document}

%% file: preamble.tex
\newcommand{\cB}{\mathcal{B}}

\newcommand{\cM}{\mathcal{M}}

\newcommand{\cV}{\mathcal{V}}

\newcommand{\bA}{\mathbf{A}}

\newcommand{\bF}{\mathbf{F}}

\DeclareMathOperator{\Tr}{Tr}

\DeclareMathOperator{\Spec}{Spec}

\DeclareMathOperator{\GL}{GL}
\DeclareMathOperator{\PGL}{PGL}
\DeclareMathOperator{\SL}{SL}
\DeclareMathOperator{\PSL}{PSL}
\DeclareMathOperator{\SO}{SO}

\DeclareMathOperator{\Res}{Res}

\newcommand*{\df}{\mathrel{\vcenter{\baselineskip0.5ex \lineskiplimit0pt
                     \hbox{\scriptsize.}\hbox{\scriptsize.}}} =}

\providecommand{\twomat}[4]{\left(\begin{matrix}#1&#2\\#3&#4\end{matrix}\right)}
\providecommand{\stwomat}[4]{\left(\begin{smallmatrix}#1&#2\\#3&#4\end{smallmatrix}\right)}
\providecommand{\twovec}[2]{\left(\begin{matrix}#1\\#2\end{matrix}\right)}

\newcommand{\QQ}{\mathbf{Q}}
\newcommand{\FF}{\mathbf{F}}
\newcommand{\CC}{\mathbf{C}}

\newcommand{\ZZ}{\mathbf{Z}}
\newcommand{\PP}{\mathbf{P}}
\newcommand{\RR}{\mathbf{R}}